\documentclass[11pt]{amsart}
\usepackage[margin=1in]{geometry}
\usepackage{graphicx} 
\usepackage{amsmath}
\usepackage{caption}
\usepackage{subfigure}
\usepackage{amsfonts}
\usepackage{dsfont}
\usepackage{amsthm}
\usepackage{geometry}
\usepackage{pgfplots}
\pgfplotsset{compat=1.18} 
\usepackage{color}
\usepackage{amsthm}
\usepackage{comment}
\usepackage[utf8]{inputenc}
\usepackage{hyperref}
\usepackage{bbm}
\usepackage{amssymb}

\newtheorem{theorem}{Theorem}[section]
\newtheorem{definition}{Definition}[section]
\newtheorem{lemma}{Lemma}[section]
\newtheorem{proposition}{Proposition}[section]

\newcommand{\Poi}{\operatorname{Poisson}}

\newcommand{\Var}{\operatorname{Var}}
\newcommand{\Cov}{\operatorname{Cov}}
\newcommand{\1}{\mathds{1}}
\newcommand{\Prob}{\mathbb{P}}
\newcommand{\Ex}{\mathbb{E}}
\title[Statistics of random-to-top shuffles]{On the statistics of random-to-top shuffles}

\author{Alexander Clay}
\date{\today}
\address{Department of Mathematics, University of Southern California (USC), Los Angeles, CA, USA.}
\email{ajclay@usc.edu}

\begin{document}
\begin{abstract}
    We prove limit theorems for the number of fixed points, descents, and inversions of iterated random-to-top shuffles in two asymptotic regimes. Our proofs are analytic, and they utilize new combinatorial decompositions that represent each statistic as a randomly indexed statistic of a uniformly random permutation. This perspective gives new combinatorial proofs of the expected number of fixed points and inversions. In particular, we solve an open problem of Pehlivan on fixed points, and we answer a question of Diaconis and Fulman on inversions. 
\end{abstract}
\maketitle

\section{Introduction}
\subsection{Background} 
What characteristics do large random permutations have? This question motivates many areas of research in probability. Large random permutations are studied using probabilistic, algebraic, and representation-theoretic techniques, with connections to mathematical physics and theoretical computer science. 

A common theme in this area is the study of probability measures on the symmetric group $S_n$, particularly those which arise through card shuffling models. An excellent survey of the mathematics of card shuffling is given in Diaconis and Fulman \cite{MathofShufflingCards}. This area has applications to Markov chains, ergodic theory, and group theory. 

One way to analyze the mixing of certain shuffles is through permutation statistics. These are numerical features of permutations. Establishing limit theorems for statistics as the size of the deck grows is a central goal in this area. From a historical standpoint, permutation statistics is a classical subject, and it has seen much recent interest in the context of card shuffling. A wonderful exposition of the work in this area can be found in both Section 4 and Section 8.4 of Diaconis and Fulman \cite{MathofShufflingCards}.

The analysis of mixing times is another interesting area of research. This subject, which originated in the 1980s, draws on Markov chain theory and the representations of the symmetric group. More specifically, the mixing time of a shuffle is equal to the number of times it must be performed before the distribution of the cards is close to uniform in a suitable metric. There are many interesting results in this area such as the ``seven shuffles theorem" of Bayer and Diaconis \cite{BayerDiaconisRiffleShuffle} for riffle shuffles, which are a common way to model how humans shuffle cards.

The connection between statistics and mixing times is explored by two questions. Suppose we are given a shuffle and a statistic. 
\begin{enumerate}
    \item How many iterations of the shuffle are necessary for the statistic to have the same asymptotic distribution as it would under a uniformly random permutation?
    \item Are there nontrivial limiting distributions of the statistic after smaller (critical) numbers of shuffles?
\end{enumerate}  

In this paper, we study the fixed points, descents, and inversions of iterated random-to-top shuffles by addressing both of these questions. Random-to-top shuffling is a classical model of great interest. Despite its importance, the statistics of random-to-top shuffles have resisted detailed analysis. Prior to our work, the limiting distributions of fixed points, descents, and inversions in any asymptotic regime were unknown. Some progress included Richard Stanley's generating function for the descents of top-to-random shuffles, as Diaconis and Fulman report \cite{MathofShufflingCards}. For random-to-top shuffling, the mixing time of the entire descent set has been studied \cite{AthDiaconis}. Meanwhile, the expected value and variance of the number of fixed points by algebraic and representation-theoretic methods \cite{Pehlivan}, and the expected number of inversions by spectral theory \cite{MathofShufflingCards} have been found. In particular, Pehlivan \cite{Pehlivan} asked for the limiting distribution of the number of fixed points in the critical regime which we study. Moreover, Diaconis and Fulman \cite{MathofShufflingCards} mentioned that there ``should be a nice theory of inversions" of random-to-top shuffles.

To motivate question (1), we note that from a computational standpoint, testing randomness via permutation statistics is more efficient than analyzing a full distribution on permutations. More precisely, we introduce the following notion. Consider a deck of $n$ cards. We say that a statistic mixes after $r=r(n)$ shuffles of the deck if, as $n\to\infty$, it has the same limiting distribution after $r=r(n)$ shuffles as it would under the uniform distribution on $S_n$. We prove that the number of fixed points, descents, and inversions mix in $r=\omega(n)$, $r=(n\log n)/2$, and $r=(n\log n)/4$ iterated random-to-top shuffles, respectively. Notably, these thresholds differ from the mixing time of random-to-top shuffling in the total variation metric, which is of order $r=n\log n$ \cite{ShufflingStopping} \cite{AnalysisTopToRandom}.

As for question (2), understanding critical limiting regimes illuminates phase transitions in the behavior of a statistic as the number of shuffles grows. We show that the number of fixed points, descents, and inversions of random-to-top shuffles each exhibit nontrivial limiting distributions in a regime where the number of shuffles is of order equal to the size of the deck. These distributions depend explicitly on the asymptotic ratio of the number of shuffles to the number of cards, which reveals phase changes in each distribution. 

We note that analogous limit theorems in critical regimes are known for the permutation statistics of other shuffles. Much is known about the statistics of riffle shuffles; for a nice exposition, see Section 4 of Diaconis and Fulman \cite{MathofShufflingCards}. Results for other types of shuffles include limit theorems for cycle counts of iterated transpositions and related shuffles \cite{Schramm} \cite{fulmanIcycles} \cite{arcona}, as well as statistics on conjugacy classes and subsets of the symmetric group \cite{KimLee} \cite{LiuYin} \cite{FulmanDescents1998} \cite{loth2023permutationstatisticsconjugacyclasses} \cite{descentsmatchings}. Our work contributes to this growing research program.

Briefly explaining our methods, we note that the main ingredient to our approach is a coupling argument which originated in the analysis of the mixing time of random-to-top shuffles by Aldous and Diaconis \cite{ShufflingStopping} and Diaconis, Fill, and Pitman \cite{AnalysisTopToRandom}. The authors in \cite{ShufflingStopping} and \cite{AnalysisTopToRandom} showed that after $r$ shuffles, the set of cards moved to the top forms a uniformly random subset whose order is also uniformly random. This insight allows us to derive distributional identities for each statistic which are crucial to the proofs of our results. We hope that these ideas will be applicable to other areas in combinatorial probability.

\subsection{Main Results}
We will now summarize our results. Let $n$ be the size of the deck and $r$ be the number of shuffles. For a constant $c>0$, we refer to the $r=cn$ regime as the ``critical" regime, and to any regime with $r\gtrsim f(n)$ as a mixed regime. Our first result shows that there is a Poisson-geometric convolution for the limiting distribution of the number of fixed points of random-to-top shuffles in the critical regime. In addition, it reveals that the mixed regime happens when the number of shuffles is asymptotically larger than the size of the deck.
\begin{theorem}
\label{fixed points main}
    Let $F_{r,n}$ be the number of fixed points after $r$ iterated random-to-top shuffles of an $n$-card deck, starting from the identity. Let $c>0$ be a constant. We have
    \[F_{cn,n}\Rightarrow_d\, X+Y\]
    where $X\sim\Poi(1-e^{-c})$ and $Y$ is geometric with parameter $1-e^{-c}$ and $P(Y=k)=(1-e^{-c})(e^{-ck})$ for all $k\geq 0$. Moreover, $X$ and $Y$ are independent. If $r=\omega(n)$, then
    \[F_{r,n}\Rightarrow_d\,\Poi(1).\]
\end{theorem}
Our second result exhibits asymptotic normality for the number of descents of iterated random-to-top in the critical regime, with the parameters depending on the ratio of the number of shuffles to the number of cards. Also, it indicates that the mixed regime occurs precisely at $(n\log n)/2 +c_n n$ shuffles, where $c_n\to\infty$. This complements the work of Athanasiadis and Diaconis \cite{AthDiaconis}, as they showed that the entire descent set of random-to-top shuffling mixes in $(n\log n)/2$ shuffles.
\begin{theorem}
\label{descents main}
    Let $D_{r,n}$ be the number of descents after $r$ iterated random-to-top shuffles of an $n$-card deck, starting from the identity. Let $c>0$ be a constant. We have
    \[\frac{D_{cn,n}-n(1-e^{-c})/2}{\sqrt{n}}\Rightarrow_d\,\mathcal{N}\left(0,\frac{1+2e^{-c}-3(1+c)e^{-2c}}{12}\right).\]
    If there exists a sequence $c_n\to\infty$ such that $r\gtrsim (n\log n)/2 + c_n n$, then
    \[\frac{D_{r,n}-n/2}{\sqrt{n}}\Rightarrow_d\,\mathcal{N}(0,1/12).\]
\end{theorem}
Our last result concerns the asymptotic normality of the number of inversions of iterated random-to-top shuffles in the critical regime, and also it indicates a threshold for the mixing regime. It is similar to the result on descents.
\begin{theorem}
\label{inversions main}
    Let $I_{r,n}$ be the number of inversions after $r$ iterated random-to-top shuffles of an $n$-card deck, starting from the identity. Let $c>0$ be a constant. We have
    \[\frac{I_{cn,n}-n^2(1-e^{-2c})/4}{n^{3/2}}\Rightarrow_d\,\mathcal{N}\left(0,\frac{1+8e^{-3c}-9(1+c)e^{-4c}}{36}\right).\]
    If there exists a sequence $c_n\to\infty$ such that $r\gtrsim (n\log n)/4 + c_n n$, then
    \[\frac{I_{r,n}-n^2/4}{n^{3/2}}\Rightarrow_d\,\mathcal{N}(0,1/36).\]
\end{theorem}
The (group-theoretic) inverse of random-to-top shuffling is top-to-random shuffling \cite{MathofShufflingCards}. Therefore, Theorem $\ref{fixed points main}$ and Theorem $\ref{inversions main}$ also apply to the fixed points and inversions of top-to-random shuffles, respectively. Unfortunately, the limiting distributions of the number of descents of top-to-random shuffles are not well-understood via our techniques. This is because a permutation and its inverse do not share the same number of descents, in general. We include a conjecture about this topic in Section $\ref{Section 8}$.
\subsection{Outline}
In Section $\ref{Section2}$, we define the permutation statistics we work with and review the necessary background material. Section $\ref{Section3}$ develops reformulations of some classical results from occupancy problems which are important to our arguments. In Section $\ref{Section 4}$, we recall the connection between card shuffling and random permutations, and we examine the combinatorics of random-to-top shuffles, establishing equalities in distribution for each statistic. Sections $\ref{Section 5}$, $\ref{Section 6}$, and $\ref{Section 7}$ motivate and prove Theorem $\ref{fixed points main}$, Theorem $\ref{descents main}$, and Theorem $\ref{inversions main}$, respectively. Finally, Section $\ref{Section 8}$ discusses several open questions and directions for future work.

\subsection{Notation}
\label{notation}
We write $[n]$ to denote the integers $\{1,2,\ldots,n\}$. We write $(n)_j$ to denote the falling factorial $n(n-1)(n-2)\cdots(n-j+1)$. For functions $f$ and $g$, we say $f(n)=\omega(g(n))$ or $f(n)\gg g(n)$ whenever  $f(n)/g(n)\to\infty$. We say that $f(n)\gtrsim g(n)$ whenever $f(n)\geq g(n)$ for all large $n$. We denote by $g(n)=O(f(n))$ sequences such that $\lim\sup_n |g(n)/f(n)|<\infty$. Write $f(n)\sim g(n)$ whenever $f(n)/g(n)\to 1$. The indicator of an event $A$ will be denoted by $\1(A)$, and an indicator random variable is denoted analogously. Permutations $\pi$ are written in one-line form $\pi=\pi(1)\pi(2)\pi(3)\ldots\pi(n)$. We write $X_n=o_p(f(n))$ to denote a sequence of random variables $X_n$ such that $X_n/f(n)\to 0$ in probability, and $X_n=O_p(f(n))$ denotes a sequence of random variables such that $X_n/f(n)$ is bounded in probability. Let $\Rightarrow_d$ denote convergence in distribution.  Denote by $F_{r,n}$, $D_{r,n}$, and $I_{r,n}$ the number of fixed points, descents, and inversions, respectively, of an $n$-card deck given $r$ iterated random-to-top shuffles, started from the identity. We denote by $K_{r,n}$ the number of occupied bins when $r$ balls are thrown into $n$ bins, independently and uniformly at random. Let $\mathcal{N}(\mu,\sigma^2)$ denote a normal random variable with mean $\mu$ and variance $\sigma^2$. We write $X=_d\, Y$ to denote equality in distribution.

\section*{Acknowledgments}
The author would like to thank his advisor, Jason Fulman, for suggesting this topic and for his references. We would also like to thank Richard Stanley, Evgeni Dimitrov, Greta Panova, and Steven Heilman, for their kind advice, and Dominic Arcona, for good discussions.
\section{Permutation Statistics}
\label{Section2}
In this section, we recall classical results from permutation statistics that will be used in the proofs of our main theorems. Readers familiar with combinatorial probability may wish to proceed directly to Section $\ref{Section3}$.

First, we define the statistics which we will work with. A fixed point is an index $i$ such that $\pi(i)=i$. The permutation $\pi=25143$ has $1$ fixed point: $\pi(4)=4$. One can think of a fixed point as a card that is in its original position after shuffling.

A descent of a permutation $\pi$ is an index $i\in[n-1]$ such that $\pi(i)>\pi(i+1)$. For example, the permutation $\pi=25143$ has descents at positions $2$ and $4$. Descents are a classical tool to test randomness.

An inversion is a pair $(i,j)$ with $1\leq i<j\leq n$ and $\pi(i)>\pi(j)$. The permutation $\pi=25143$ has $5$ inversions: the pairs $(1,3)$, $(2,3)$, $(2,4)$, $(2,5)$, and $(4,5)$. Inversions test how out of order a permutation is.

Under the uniform distribution on permutations, the limiting distributions of the number of fixed points, descents, and inversions are well known, classical results. We state these results here for later use.
\begin{proposition}
\label{uniform case convergences}
    Let $\pi$ be a uniformly random permutation of $[n]$. Let $F(\pi)$ be the number of fixed points, $d(\pi)$ be the number of descents, and $I(\pi)$ be the number of inversions. Then, we have
    \begin{equation} 
    \label{ordinary fixed points}
    F(\pi)\Rightarrow_d\; \Poi(1),
    \end{equation}
    \begin{equation}
    \label{ordinary descents}
        \frac{d(\pi)-n/2}{\sqrt{n}}\Rightarrow_d\; \mathcal{N}(0,1/12),\;\;\;\; \text{and}
    \end{equation}
    \begin{equation}
    \label{ordinary inversions}
        \frac{I(\pi)-n^2 / 4}{n^{3/2}}\Rightarrow_d\; \mathcal{N}(0,1/36).
    \end{equation}
\end{proposition}
\begin{proof}
    We will sketch the proofs, as similar techniques will be used in later sections. Throughout, we assume that $\pi$ is a uniformly random permutation of $[n]$.
    
    Let $F(\pi)$ be the number of fixed points. By inclusion–exclusion, one obtains
    \[\Prob(F(\pi)=k)=\frac{1}{k!}\sum_{j=0}^{n-k}\frac{(-1)^j}{j!}\to \frac{e^{-1}}{k!}.\]
    This implies the convergence in Equation $\eqref{ordinary fixed points}$.
    
    Let $d(\pi)$ be the number of descents. One has the decomposition
    \begin{equation}
        \label{decomposition of uniformly random descents}
        d(\pi)=\sum_{i=1}^{n-1} \1(\pi(i)>\pi(i+1))=_d\,\sum_{i=1}^{n-1}\1(U_{i}>U_{i+1})
    \end{equation}
    where the $(U_i)_{i\geq 1}^{n}$ are i.i.d., that is, independent and identically distributed, continuous uniform $(0,1)$ random variables. We note that the equality in distribution follows from the standard method of generating uniformly random permutations with the order statistics of i.i.d. continuous distributions. Let $X_i=\1(U_i>U_{i+1})$. It is straightforward to verify that the random variables $(X_i)_{i\geq 1}^{n-1}$ form a one-dependent sequence, i.e. when $|i-j|>1$, we have that $X_i$ and $X_j$ are independent. The central limit theorem for $m$-dependent random variables, as in \cite[Theorem 1]{m-dependentclt-hoeffding}, implies the convergence in Equation $\eqref{ordinary descents}$. 

    Finally, let $I(\pi)$ be the number of inversions. For each $i\in[n-1]$, set
    \begin{equation}
    \label{inversions uniform decomposition}
        R_i=\sum_{k=i}^n \1(\pi(i)>\pi(k)).
    \end{equation}
    One checks that 
    \[I(\pi)=_d\,\sum_{i=1}^{n} R_i\]
    and that the $(R_i)_{i\geq 1}^n$ are independent discrete uniform random variables, each supported on $\{0,1,2,$ $\ldots,n-i\}$. The convergence in Equation $\eqref{ordinary inversions}$ then follows by Lindeberg-Feller as in \cite[Theorem 3.4.10]{durrett2019probability}.
\end{proof}
We note that the results in Proposition $\ref{uniform case convergences}$ have been substantially refined over the years. Indeed, there is a large literature on the statistics of uniformly random permutations. It is beyond the scope of this paper to elaborate on the treasure trove of results. We refer the reader to Arratia and Tavar\'e \cite{ArratiaTavare1992}, Diaconis and Chatterjee \cite{ChatterjeeDiaconisDescents}, and Margolius \cite{Margolius}, for more recent studies into fixed points, descents, and inversions, respectively. More specifically, we mention that precise convergence rates to normality for each of the three statistics in Proposition $\ref{uniform case convergences}$ are known. In particular, Fulman \cite{FulmanDescents} uses Stein's method with an explicit connection to card shuffling to establish a convergence rate in Equation $\eqref{ordinary descents}$ of Proposition $\ref{uniform case convergences}$.

To conclude this section, we will state two important lemmas concerning convergence in distribution. These will eventually help us apply the results from Proposition $\ref{uniform case convergences}$ to iterated top-to-random shuffles. The following lemma is known as Slutsky's Theorem or the convergence-together lemma, and is found in \cite[Exercise 3.2.13]{durrett2019probability}. We leave out the proof.
\begin{theorem}[Slutsky] 
\label{slutsky}
Suppose $X_n$ and $Y_n$ are sequences of random variables with $X_n\Rightarrow_d X$ and $Y_n\Rightarrow_d a$, where $a$ is a constant random variable. Then, $X_n+Y_n\Rightarrow_d\, X+a$.
\end{theorem}
The next lemma from \cite[Theorem 2.7]{AsymptoticStatistics} will allow us to compare randomly indexed and deterministically indexed models of permutation statistics. We omit the proof.
\begin{lemma}
\label{slutsky ii}
    Suppose $X_n$ and $Y_n$ are sequences of random variables such that $|X_n-Y_n|\to 0$ in probability and $X_n\Rightarrow_d X$. Then $Y_n\Rightarrow_d X$.
\end{lemma}
This completes the background material needed for the remainder of the paper.

\section{Balls in Boxes}
\label{Section3}
In this section, we study the distribution of the random variable $K_{r,n}$, which counts the number of occupied boxes when $r$ balls are thrown independently and uniformly into $n$ boxes. Later in this paper, we will exploit our knowledge of the random variable $K_{r,n}$ together with the combinatorial decompositions in Section $\ref{Section 4}$ to prove our main results. The occupancy (or coupon collector) problem has been studied in great generality, and it is beyond the scope of this paper to survey the many results. A good reference in this area is Ferrante and Saltalamacchia \cite{Ferrante}. We will provide reformulations of a few classical results of Weiss \cite{Weiss1958Limiting}, who studied the number of \textit{un}occupied boxes in the same setup.
\begin{proposition}
\label{facts about krn}
    Let $K_{r,n}$ be the number of occupied boxes when $r$ balls are thrown into $n$ boxes uniformly at random. Then, we have
    \begin{enumerate}
        \item $\Ex[K_{r,n}]=n-n(1-1/n)^r$ 
        \item $\Var K_{r,n} = n(1-1/n)^r +n(n-1)(1-2/n)^r -n^2(1-1/n)^{2r}\;\;\;\\ \text{and}\;\;\;\Var K_{cn,n}\sim n(e^{-c}-(1+c)e^{-2c})$.
        \item We have
        \[\frac{K_{cn,n}- n(1-e^{-c})}{\sqrt{n(e^{-c}-(1+c)e^{-2c})}}\Rightarrow_d\,\mathcal{N}(0,1). \]
    \end{enumerate}
\end{proposition}
\begin{proof}
    Adopting the methods from Weiss \cite{Weiss1958Limiting}, we first prove (1). Let $\widetilde{K}_{r,n}$ be the number of \textit{un}occupied boxes in this setup. Since $K_{r,n}+\widetilde{K}_{r,n}=n$, we have $\Ex[K_{r,n}]=n-\Ex[\widetilde{K}_{r,n}]$. Let $X_i$ be the indicator that box $i$ is unoccupied. Then $\Ex[X_i]=(1-1/n)^r$, so that
    \[\Ex[K_{r,n}]=n-\sum_{i=1}^n \Ex[X_i]=n-n(1-1/n)^r.\]
    To show (2), we will do a covariance expansion. First, notice that $\Var{K_{r,n}}=\Var\widetilde{K}_{r,n}$. We have
    \[\Var \widetilde{K}_{r,n}=\sum_{i=1}^n\Var X_i+2\sum_{i<j}\text{Cov}(X_i,X_j)\]
    We can compute $\Var X_i=(1-1/n)^r-(1-1/n)^{2r}$, and, for $i\neq j$, 
    \begin{equation*}
        \begin{aligned}
            \text{Cov}(X_i,X_j) &=\Prob(\text{boxes\,}i\text{ and }j \text{ both unoccupied})-(1-1/n)^{2r}\\
            &=(1-2/n)^{r}-(1-1/n)^{2r}
        \end{aligned}
    \end{equation*}
    so that 
    \begin{equation*}
        \begin{aligned}
            \Var K_{r,n} &=\Var \widetilde{K}_{r,n} \\
            &=n((1-1/n)^r-(1-1/n)^{2r})+n(n-1)((1-2/n)^r-(1-1/n)^{2r})\\
            &=n(1-1/n)^r +n(n-1)(1-2/n)^r -n^2(1-1/n)^{2r}
        \end{aligned}
    \end{equation*}
    as desired. The stated asymptotics follow by taking $r=cn$ and using $(1-a/n)^{cn}\to e^{-ac}$ for any constants $a$ and $c$ where $c>0$. For (3), Weiss \cite{Weiss1958Limiting} proves that
    \[\frac{\widetilde{K}_{cn,n}-ne^{-c}}{\sqrt{n(e^{-c}-(1+c)e^{-2c})}}\Rightarrow_d \,\mathcal{N}(0,1).\]
    To transfer this to the convergence in (3) with $K_{cn,n}$, notice that
    \[\frac{K_{cn,n}-(n-ne^{-c})}{\sqrt{n(e^{-c}-(1+c)e^{-2c})}}=\frac{(n-\widetilde{K}_{cn,n})-(n-ne^{-c})}{\sqrt{n(e^{-c}-(1+c)e^{-2c})}}\]
    \[=-\,\frac{\widetilde{K}_{cn,n}-ne^{-c}}{\sqrt{n(e^{-c}-(1+c)e^{-2c})}}\Rightarrow_d\,\mathcal{N}(0,1)\]
    where we used both the continuous mapping theorem with $f(x)=-x$, and the elementary fact that when $Z$ is a standard normal random variable, then $-Z$ is as well.
\end{proof}
To conclude, we state the probability mass function of $K_{r,n}$ and note a connection between occupancy problems and Stirling numbers. Let $S(r,k)$ denote the Stirling numbers of the second kind, that is, the number of ways to divide a set of $r$ objects into $k$ non-empty subsets. We have
\[\Prob(K_{r,n}=k)=\frac{\binom{n}{k} k! S(r,k)}{n^r},\]
which follows from choosing $k$ boxes to be occupied, permuting the boxes, and then dividing the $r$ balls among each of the $k$ boxes.

\section{Combinatorics of Random-to-Top Shuffles}
\label{Section 4}
This section recalls the setup of iterated random-to-top shuffles and investigates their combinatorial structure. We will use the notation for $K_{r,n}$, $F_{r,n}$, $D_{r,n}$, and $I_{r,n}$ as in Section $\ref{notation}$. The goal of this section is to prove the following three equalities in distribution. 
\begin{theorem}
\label{decomposition of fixed points}
Let $\pi$ be a uniformly random permutation of $[n]$, independent of $K_{r,n}$. We have
\[F_{r,n} =_d \; n-\max_{1\leq i\leq K_{r,n}} (\pi(i))+\sum_{i=1}^{K_{r,n}} \1(\pi(i)=i).\]
\end{theorem}
\begin{theorem}
\label{decomposition of descents}
Let $(U_i)_{i\geq 1}^n$ be i.i.d. continuous uniform$(0,1)$ random variables, independent of $K_{r,n}$. We have
\[D_{r,n} =_d \; O_p(1)+\sum_{i=1}^{K_{r,n}-1}\1(U_i>U_{i+1}) \]
where the $O_p(1)$ term is bounded by $1$.
\end{theorem}
\begin{theorem}
\label{decomposition of inversions}
    We have 
    \[I_{r,n}=_d \sum_{i=1}^{K_{r,n}} R_i\]
    where the $R_i$ are independent discrete uniform random variables each supported on $\{0,1,2,\ldots,n-i\}$ and are independent of $K_{r,n}$.
\end{theorem}
We outline the approach to these decompositions as follows. We will decompose permutations obtained by iterated random-to-top shuffling into sub-structures that depend on the number of distinct cards which have moved to the top, which is equal in distribution to $K_{r,n}$. Then, we will exhibit bijections between these sub-structures with subsets of uniformly random permutations, and we will take their statistics. The distributional equivalences in Theorem $\ref{decomposition of fixed points}$, Theorem $\ref{decomposition of descents}$, and Theorem $\ref{decomposition of inversions}$ are powerful because they allow us to exploit well-known results on balls-in-bins and permutation statistics to prove our main results. 

To begin motivating our theory, we will first define the connection between card shuffling and permutations. We can illustrate this relationship as follows. Label a deck of cards $1,2,\ldots,n$. Given a probability $\mathbb{P}$ on the symmetric group $S_n$, pick a permutation $\pi$ from $\mathbb{P}$. Then $\pi(i)$ corresponds to the card in position $i$ of the shuffled deck. For example, if we pick $\pi=25143$, then $\pi(1)=2$, so card $2$ is in position $1$, $\pi(2)=5$, so card $5$ is in position $2$, and so on.

We will consider iterated shuffling as a random walk on the symmetric group. Specifically, we start with the identity permutation $id=12\ldots n$. Given a probability $\mathbb{P}$ on $S_n$ corresponding to a shuffle, we can perform $r$ iterated shuffles by picking permutations $\pi_1,\pi_2,\ldots,\pi_r$ independently from $\mathbb{P}$ and multiplying them together on the left. This forms the product
\[\pi_r\pi_{r-1}\cdots\pi_1\cdot id=\pi_r\pi_{r-1}\cdots\pi_1.\]
Here, multiplication is given as composition in the symmetric group $S_n$, where, for permutations $\pi$ and $\sigma$, we have $(\pi\cdot\sigma)(i)=\pi(\sigma(i))$. One can check that this definition of iterated shuffling is compatible with our above definition of a single shuffle. We remark that the induced probability measure on $S_n$ corresponding to $r$ iterated shuffles from a given measure $\mathbb{P}$ is calculated by performing an $r$-fold convolution of $\mathbb{P}$ with itself. This concept is explained in detail in Chapter $1$ of Diaconis and Fulman \cite{MathofShufflingCards}, but we will not need it to obtain our results. Indeed, for random-to-top shuffles, iterated shuffling admits a rich combinatorial description. We will show in this section that such a characterization is sufficient for our analysis. The first step is to recall the definition of random-to-top shuffling which is consistently used in the literature, and which we will follow.
\begin{definition}
    A random-to-top shuffle is performed by selecting a random card and moving it to the top of the deck. In particular, we allow the top card to be selected, in which the ordering of the deck does not change. Iterated random-to-top shuffles start with the deck in order $12\cdots n$.
\end{definition}
For example, if the deck has $n=8$ cards and we start with the identity permutation $\pi=12\cdots8$ and perform a random-to-top shuffle with card $4$ selected, the resulting permutation is $41235678$. If we again start with $\pi=12\cdots8$ and we perform $r=3$ iterated random-to-top shuffles, with cards $3,6,1$ selected in order, then the resulting permutation is $16324578$. 

Now, we will show the distributional equivalence which allows us to conditionally sample the permutation statistics of random-to-top shuffles from independent uniformly random permutations. Recall that a $k$-permutation of $[n]$ is defined as a permutation of a subset of $k$ distinct integers of $[n]$. For example, $\sigma=3285$ is a $4$-permutation of $[11]$.  
\begin{proposition}
\label{resampling uniform}
    Suppose that after $r$ random-to-top shuffles of an $n$-card deck, exactly $k$ unique cards have been moved to the top. Then, the cards which have been selected are in the first $k$ positions of the deck, and they form a uniformly random $k$-permutation of $[n]$. The cards that have not been selected are in the last $n-k$ positions of the deck in the same relative (increasing) order.
\end{proposition}
\begin{proof}
    Suppose that exactly $k$ unique cards have been moved to the top after $r$ random-to-top shuffles. The cards which have been moved are a uniformly random $k$-subset of $[n]$ since all cards are equally likely to be moved to the top. Now the ordering of the $k$ cards which have been moved is a uniformly random permutation of the uniformly random $k$-subset of $[n]$ which we chose, by \cite{AnalysisTopToRandom}. Therefore, the first $k$ cards are equal in distribution to the first $k$ positions of an independently chosen uniformly random permutation. It is easy to see that the $n-k$ cards which have not been moved to the top are in the same relative order at the bottom of the deck.
\end{proof}
We can equate in distribution the number of distinct cards which have been moved to the top of the deck after $r$ to the occupancy random variables $K_{r,n}$ from Section $\ref{Section3}$.
\begin{lemma}
    The random variable $K_{r,n}$ is equal in distribution to the number of distinct cards which have been moved to the top after $r$ top-to-random shuffles of an $n$-card deck.
\end{lemma}
\begin{proof}
    We establish a bijection as follows. An assignment of a ball to a box is the same as choosing a card and moving it to the top of the deck. The number of boxes which are occupied is then equal to the number of distinct cards which have been moved to the top.
\end{proof}
We now have all of the ingredients to prove the decompositions in Theorem $\ref{decomposition of fixed points}$, Theorem $\ref{decomposition of descents}$, and Theorem $\ref{decomposition of inversions}$. We start with the number of fixed points.
\begin{proof}[Proof of Theorem $\ref{decomposition of fixed points}$]
 By Proposition $\ref{resampling uniform}$, we can consider a uniformly random permutation $\pi$, independent of $K_{r,n}$. Then, reorder the elements in the last $n-K_{r,n}$ positions of $\pi$ in increasing order to form a permutation $\pi'$. This induces the same conditional distribution on $r$ iterations of random-to-top shuffles when exactly $K_{r,n}$ distinct cards have been moved to the top. Count the number of fixed points in the first $K_{r,n}$ positions of $\pi'$. This is equal to the number of fixed points in the first $K_{r,n}$ positions of $\pi$. Then, count the fixed points in the remaining $n-K_{r,n}$ positions of $\pi'$. This quantity is equal to the length of the string of \textit{consecutive} cards $\{n-i,n-i+1,\ldots,n\}$ which ends in $n$ at the end of the permutation $\pi'$. This is equal to $n-\max_{1\leq i\leq K_{r,n}}\pi(i)$, and in particular, is equal to zero if and only if card $n$ has been moved to the top. One obtains the formula in Theorem $\ref{decomposition of fixed points}$.
\end{proof}

For example, suppose $n=9$, $r=7$, and $K_{r,n}=5$, and we sample a uniformly random permutation 
\begin{equation}
\label{sampling fpts}
    \pi=273519684.
\end{equation}
The corresponding reordered permutation corresponding to iterated random-to-top shuffles, conditional on $K_{r,n}=5$ distinct cards being moved to the top, is given by 
\begin{equation}
\label{sampling fpts ii}
    \pi'=\underbrace{27351}_{\text{length } K_{r,n}}4689.
\end{equation}
There is one fixed point in position $3$ from the first $K_{r,n}$ elements of $\pi$, and the maximal element in the first $5$ positions of $\pi$ is equal to $7$. The number of fixed points of $\pi'$ is equal to $3$, which agrees with the quantity from Theorem $\ref{decomposition of fixed points}$.

We next prove the decomposition of the number of descents.
\begin{proof}[Proof of Theorem $\ref{decomposition of descents}$]
We use a similar argument to the proof of Theorem $\ref{decomposition of fixed points}$. Proposition $\ref{resampling uniform}$ allows us to consider a uniformly random permutation $\pi$, independent of $K_{r,n}$. Then, reorder the elements in the last $n-K_{r,n}$ positions of $\pi$ in ascending order to form a permutation $\pi'$. Since the last $n-K_{r,n}$ elements of $\pi'$ are in ascending order, there are no descents in these positions. The number of descents in the first $K_{r,n}$ positions of $\pi'$ is equal to the number of descents in the first $K_{r,n}$ positions of $\pi$. Since $\pi$ is uniformly random, we have
\[\sum_{i=1}^{K_{r,n}-1}(\pi(i)>\pi(i+1))=_d\sum_{i=1}^{K_{r,n}-1}\1(U_i>U_j)\]
where the $(U_i)_{i\geq 1}^n$ are i.i.d. uniform$(0,1)$ random variables independent of $K_{r,n}$, and the equality in distribution follows from Equation $\eqref{decomposition of uniformly random descents}$.
It remains to count whether $\pi(K_{r,n})>\pi'(K_{r,n}+1)$. This term contributes at most one descent, so it is $O_p(1)$.
\end{proof}

Working with the example from Equation $\eqref{sampling fpts}$ and Equation $\eqref{sampling fpts ii}$, we find that the number of descents in the first $K_{r,n}=5$ positions of $\pi$ is equal to $2$. The last thing we need to check is the corresponding $O_p(1)$ term, and this is equal to the indicator that $\pi(5)>\pi'(6)$, which contributes zero in this case. So, in this example, the corresponding number of descents after iterations of random-to-top shuffles, conditional on $K_{r,n}$, is equal to $2$.

Lastly, we prove the decomposition of the number of inversions.
\begin{proof}[Proof of Theorem $\ref{decomposition of inversions}$]
    Suppose that exactly $K_{r,n}$ distinct cards have been moved to the top. As before, let $\pi$ be a uniformly random permutation of $[n]$. Reorder the last $n-K_{r,n}$ elements of $\pi$ into increasing order to form a new permutation $\pi'$. Then $\pi'$ has the same distribution as a permutation obtained by $r$ iterated random-to-top shuffles, conditional on $K_{r,n}$ distinct cards being moved to the top. Since the last $n-K_{r,n}$ cards of $\pi'$ are in increasing order, we only need to count inversions $(i,j)$ of $\pi'$ such that $i\leq K_{r,n}$ (with no restrictions on $j$ besides $j>i$). But this is equal to the number of inversions $(i',j')$ of $\pi$ such that $i'\leq K_{r,n}$. The decomposition follows.
\end{proof}

Within the example from Equation $\eqref{sampling fpts}$ and Equation $\eqref{sampling fpts ii}$, we clearly see that all of the inversions $(i,j)$ of $\pi'$ have $i\leq 5$, and $K_{r,n}=5$ in this case. This concludes our analysis of the equalities in distribution of each statistic.

\section{Fixed Points}
\label{Section 5}
We begin by giving a combinatorial proof of the expected number of fixed points, based on the return probabilities of individual cards. We then develop the additional ingredients needed to prove Theorem~\ref{fixed points main} and conclude with the proof of that result.

In her PhD thesis, Pehlivan \cite{Pehlivan} gave three proofs of the formula for the expected number of fixed points after iterated random-to-top shuffles. She showed that
\begin{equation}
\label{expected fixed points Pehlivan}
    \Ex[F_{r,n}]=1+\sum_{k=0}^{n-2}\left(\frac{k}{n}\right)^r.
\end{equation}
As summarized by Diaconis and Fulman \cite{MathofShufflingCards}, Pehlivan's first proof used an explicit formula for the chance of a permutation after $r$ shuffles, her second proof used the eigenvalues of the transition matrix of a single shuffle, and her third used representation theory. The simplicity of Equation \eqref{expected fixed points Pehlivan} suggests that there should also be a direct combinatorial argument. We obtain such a proof by analyzing the return probability of each card to its original position.
\begin{proposition}
\label{return probability}
Let $W_{r,k}$ be the event that, after $r$ random-to-top shuffles, card $k$ is in position $k$, i.e. $k$ is a fixed point of the permutation corresponding to the shuffled deck. Then
\[
\Prob(W_{r,k})=\left(\frac{k-1}{n}\right)^r+\frac{\Prob(K_{r,n}\ge k)}{n},
\]
where $K_{r,n}$ is the number of occupied boxes when $r$ balls are thrown independently and uniformly into $n$ boxes.
\end{proposition}
\begin{proof}
Let $A_{r,k}$ be the event that at least one of the cards in $\{k,k+1,\ldots,n\}$ has been moved to the top during the $r$ shuffles. If no card in $\{k,k+1,\ldots,n\}$ is ever selected, then card $k$ is certainly fixed. Since each shuffle selects one of the first $k-1$ cards with probability $(k-1)/n$, we have that $(A_{r,k})^c$ has probability $((k-1)/n)^r$.

Now suppose that $A_{r,k}$ occurs. On this event, card $k$ can be fixed only if at least $k$ distinct cards have been moved to the top, that is, only if $K_{r,n}\geq k$. Conditional on $\{K_{r,n}\geq k\}$, Proposition~\ref{resampling uniform} implies that the first $k$ positions of the shuffled deck are distributed as the first $k$ positions of a uniformly random permutation of $[n]$. Therefore, conditional on $\{K_{r,n}\ge k\}$, the probability that position $k$ is a fixed point is $1/n$. Spelling this out, we have the following tower of conditional probabilities.
    \begin{equation}
    \begin{aligned}
        \Prob(W_{r,k}) &  =\Prob(W_{r,k}\mid A_{r,k})\cdot\Prob(A_{r,k})+\Prob(W_{r,k}\mid (A_{r,k})^c)\cdot\Prob((A_{r,k})^c) \\
        & =\Prob(W_{r,k}\mid A_{r,k})\cdot\Prob(A_{r,k})+\left(\frac{k-1}{n}\right)^r \label{conditional prob fixed points}
        \end{aligned}
    \end{equation}
    By our combinatorial reasoning, we have
    \begin{equation}
    \label{conditional prob expansion}
        \begin{aligned}
            \Prob(W_{r,k}\mid A_{r,k}) &=\Prob(W_{r,k}\cap \{K_{r,n}\geq k\}\mid A_{r,k}) \\
            &=\Prob(W_{r,k}\mid \{K_{r,n}\geq k\}\cap A_{r,k})\cdot\Prob(K_{r,n}\geq k\mid A_{r,k}) \\
            &=\frac{1}{n}\cdot\Prob(K_{r,n}\geq k\mid A_{r,k}) \\
            & =\frac{1}{n}\cdot\frac{\Prob(\{K_{r,n}\geq k\}\cap A_{r,k})}{\Prob(A_{r,k})} \\
            & =\frac{1}{n}\cdot\frac{\Prob(K_{r,n}\geq k)}{\Prob(A_{r,k})}
        \end{aligned}
    \end{equation}
    where the last equality follows from the pigeonhole principle, since whenever at least $k$ distinct cards have been selected, it is necessary that at least one of the cards $\{k,k+1,\ldots, n\}$ has been selected. This implies that $\{K_{r,n}\geq k\}\subseteq A_{r,k}$. Putting Equation $\eqref{conditional prob expansion}$ together with Equation $\eqref{conditional prob fixed points}$, we obtain the proposition.
\end{proof}
We will now use Proposition $\ref{return probability}$ to recover the equality in Equation $\eqref{expected fixed points Pehlivan}$. Let $X_{r,k}$ be the indicator that card $k$ is in position $k$ after $r$ top-to-random shuffles started from the identity. We have
\[F_{r,n}=\sum_{k=1}^n X_{r,k},\]
so taking expectations of both sides and applying Proposition $\ref{return probability}$,
\begin{equation*}
\begin{aligned}
    \Ex[F_{r,n}] 
             &=\sum_{k=1}^n \Prob(W_{r,k}) \\
             &=\sum_{k=1}^n\left(\frac{\Prob(K_{r,n}\geq k)}{n}+\left(\frac{k-1}{n}\right)^r \right) \\
            &=\frac{\Ex[K_{r,n}]}{n}+\sum_{k=1}^n \left(\frac{k-1}{n}\right)^r \\
            &=1-\left(\frac{n-1}{n}\right)^r+\sum_{k=1}^n \left(\frac{k-1}{n}\right)^r \\
            & =1+\sum_{k=0}^{n-2}\left(\frac{k}{n}\right)^r.
    \end{aligned}
\end{equation*}
This gives a combinatorial proof of Pehlivan's expected value.

Pehlivan further showed that when $r=cn$ for a constant $c>0$ and $n\to\infty$,
\[
\Ex[F_{cn,n}]\to 1-e^{-c}+\frac{1}{e^c-1},
\]
and
\[
\Var F_{cn,n}\to 1-e^{-c}+\frac{e^c}{(e^c-1)^2}.
\]
In particular, these limits are not equal, so $F_{cn,n}$ does not converge in distribution to a Poisson random variable.

To motivate our characterization of the limiting distribution of $F_{cn,n}$, it is helpful to revisit the decomposition from Theorem $\ref{decomposition of fixed points}$. Recall that
\[
F_{r,n}=_d\, n-\max_{1\le i\le K_{r,n}}\pi(i)+\sum_{i=1}^{K_{r,n}}\1(\pi(i)=i),
\]
where $\pi$ is a uniformly random permutation of $[n]$ independent of $K_{r,n}$. By Proposition~\ref{facts about krn}, when $r=cn$, we have
\[
\Ex[K_{cn,n}]=n\Bigl(1-(1-1/n)^{cn}\Bigr)\sim n(1-e^{-c}).
\]
This suggests replacing the random index $K_{cn,n}$ by the deterministic value $\lfloor n(1-e^{-c})\rfloor$. Under that heuristic, one expects the deterministically indexed random variables

\[n-\max_{1\leq i\leq \lfloor n(1-e^{-c})\rfloor} (\pi(i))\]
and 
\[\sum_{i=1}^{\lfloor n(1-e^{-c})\rfloor} \1(\pi(i)=i)\] to converge to a zero-indexed geometric random variable and a Poisson random variable, respectively, both with parameter $1-e^{-c}$. More precisely, if $X\sim \Poi(1-e^{-c})$ and $Y$ is a zero-indexed geometric random variable with parameter $1-e^{-c}$, and $X$ and $Y$ are independent, then
\[
\Ex[X+Y]=1-e^{-c}+\frac{1}{e^c-1}
\]
and
\[
\Var(X+Y)=1-e^{-c}+\frac{e^c}{(e^c-1)^2},
\]
matching the limits computed by Pehlivan \cite{Pehlivan}. This was the main clue that the decomposition in Theorem $\ref{decomposition of fixed points}$ was the correct starting point.

Our proof of Theorem $\ref{fixed points main}$ proceeds in two steps. First, we prove convergence for the model obtained by replacing $K_{cn,n}$ in $F_{cn,n}$ with $\lfloor a_n n\rfloor$, where $a_n$ is a deterministic sequence taking values in $(0,1]$ which converges uniformly to a limit $a\in(0,1]$. Then, we show that the random index $K_{cn,n}$ may be replaced by its mean without changing the limiting distribution.

To begin developing the material for the proof, we will need a few technical lemmas. The first lemma will allow us to match the limiting distributions.

\begin{lemma}
    Let $X\sim\Poi(a)$ and let $Y$ be geometric with $P(Y=k)=a(1-a)^k$. Suppose $X$ and $Y$ are independent. Then, we have 
    \begin{equation}
    \label{pmf of x+y}   
        \Prob(X+Y=\ell)=\sum_{j=0}^{\ell}\frac{a(1-a)^j e^{-a}a^{\ell -j}}{(\ell -j)!}
    \end{equation}
\end{lemma}
\begin{proof}
    This is a direct convolution computation.
\end{proof}
The next lemma concerns the distribution of the maximum of the first given number of positions of a random permutation.
\begin{lemma}
\label{max of uniformly random}
    Let $\pi$ be a uniformly random permutation of $[n]$. We have
    \[\Prob\left(\max_{1\leq i\leq j}\pi(i)=m\right)=\frac{\binom{m-1}{j-1}}{\binom{n}{j}}.\]
\end{lemma}
\begin{proof}
    The first $j$ entries of a uniformly random permutation are a uniformly random ordered $j$-subset of $[n]$. The event that their maximum is $m$ occurs exactly when $m$ is selected and the other $j-1$ entries are chosen from $[m-1]$. There are $\binom{m-1}{j-1}$ such choices out of $\binom{n}{j}$ total.
\end{proof}
We also need a counting lemma for the number of fixed points in a truncated permutation. We define fixed points of $k$-permutations analogously to those of permutations. For example, $\sigma=3285$, considered as a $4$-permutation of $[11]$, has one fixed point, since $\sigma(2)=2$.
\begin{lemma}
Let $Q(k,m,s)$ be the probability that a random $k-1$-permutation of $[m-1]$ has exactly $s$ fixed points. We have
\begin{equation}
\label{formula for qkms}
Q(k,m,s)=\frac{1}{s!}\sum_{t=s}^{k-1}\frac{(-1)^{t-s}(k-1)_t}{(t-s)!(m-1)_t}.
\end{equation}
\end{lemma}
\begin{proof}
    We use the principle of inclusion-exclusion to calculate $Q$. The total number of possible permutations is $(m-1)!/(m-k)!$. To calculate the number of permutations with at least $t$ fixed points, first choose $t$ elements from $[k-1]$ to fix. Then, permute $k-t-1$ of the $m-t-1$ remaining elements. There are $(m-t-1)!/(m-k)!$ ways to do this. Hence,
    \[Q(k,m,s)=\sum_{t=s}^{k-1}(-1)^{t-s}\binom{t}{s}\binom{k-1}{t}\frac{(m-t-1)!(m-k)!}{(m-k)!(m-1)!} \]
    which, after simplifying, equals the expression for $Q(k,m,s)$ in Equation $\eqref{formula for qkms}$.
\end{proof}
We have enough tools to prove the convergence in distribution of the deterministically-indexed model.
\begin{proposition}
\label{prop convergence of deterministic}
    Let $\pi$ be a uniformly random permutation of $[n]$ and let $0<a\leq 1$ be a constant. Then, we have
    \begin{equation}
    \label{convergence of deterministic model}
    L_{a,n}:=n-\max_{1\leq i\leq \lfloor an \rfloor} (\pi(i))+\sum_{i=1}^{\lfloor an \rfloor} \1(\pi(i)=i)\Rightarrow_d X+Y
    \end{equation}
    where $X$ and $Y$ are independent random variables with $X\sim\Poi(a)$ and $Y$ is a geometric random variable with parameter equal to $a$ and $P(Y=k)=a(1-a)^k$. If $a_n\in(0,1]$ is a sequence with $a_n\to a\in(0,1]$ as $n\to\infty$ such that for some $\alpha>0$, we have $a_n>\alpha>0$ for all $n$, then $L_{a_n,n}$ has the same limiting distribution as $L_{a,n}$ in equation $\eqref{convergence of deterministic model}$.
\end{proposition}
\begin{proof}
    Throughout the proof, we let $Q(k,m,s)$ be the probability that a random $(k-1)$-permutation of $[m-1]$ has exactly $s$ fixed points as in Equation $\ref{formula for qkms}$.
    
    We prove the result for a constant parameter $a$ first. Let us calculate the probability mass function of $L_{a,n}$. Let $M_{an}=\max_{1\leq i\leq \lfloor an \rfloor} (\pi(i))$. We will condition on the value of $M_{an}$. By Lemma $\ref{max of uniformly random}$, we have
    \begin{equation}
    \label{distribution of Ma_n}
    \Prob(M_{an}=m)=\frac{\binom{m-1}{\lfloor an\rfloor-1}}{\binom{n}{\lfloor an\rfloor}}.
    \end{equation}
    Set 
    \[H_{an}=\sum_{i=1}^{\lfloor an \rfloor} \1(\pi(i)=i).\]
    Fix an integer $\ell\geq0$. We want to find the conditional distribution $\Prob(H_{an}=\ell -n+m\mid M_{an} =m)$. This boils down to solving the following counting problem. Given that the maximum of a random $\lfloor an\rfloor$-permutation $\pi'$ of $[n]$ is equal to $m$, we need to find the probability that $\pi'$ has $\ell-n+m$ fixed points. Note that when $m=\lfloor an\rfloor$, then by assumption $\pi'$ has maximum element $\lfloor an\rfloor$. By the pigeonhole principle, $\pi'$ is a uniformly random permutation of $\lfloor an\rfloor$ in this case. Therefore, it is immediate that
    \begin{equation}
    \label{m equals an case}
    \Prob(H_{an}=\ell -n+\lfloor an\rfloor\mid M_{an} =\lfloor an\rfloor)=Q(\lfloor an\rfloor,\lfloor an\rfloor,\ell-n+\lfloor an\rfloor).
    \end{equation}
    On the other hand, suppose that $m>\lfloor an\rfloor$. Set $k=\lfloor an\rfloor$, and let $\pi'$ be a uniformly random $k$-permutation of $[n]$ conditioned to have $m$ as its maximal element. Since $m>k$, the position of $\pi'$ in which $m$ is placed cannot be a fixed point. Without loss of generality, we may assume that $\pi'(k)=m$. This is because all of the positions of $\pi'$ are equally likely to be fixed, and $m$ is equally likely to be in any position. Now only elements of $[m-1]$ can be placed in the remaining first $k-1$ positions of $\pi'$ by assumption, since $m$ is the maximal element. Therefore, we can count the number of fixed points of a random $(k-1)$-permutation of $[m-1]$.  Combining this reasoning with Equation $\ref{m equals an case}$, we have, for any $m\geq \lfloor an\rfloor$, 
    \begin{equation}
    \label{distribution of phi^a_n}
    \Prob(H_{an}=\ell-n+m\mid M_{an}=m)=Q(\lfloor an\rfloor,m,\ell-n+m).
    \end{equation}
    By the law of total probability, Equation $\eqref{distribution of phi^a_n}$, and Equation $\eqref{distribution of Ma_n}$, we have
    \begin{equation*}
        \begin{aligned}
            \Prob(L_{a,n}=\ell)& =\sum_{m=\lfloor an\rfloor}^{n}Q(\lfloor an\rfloor,m,\ell -n+m)\cdot \Prob(M_{an,n}=m) \\
                            & =\sum_{m=\lfloor an\rfloor}^{n}Q(\lfloor an\rfloor,m,\ell -n+m)\cdot\frac{\binom{m-1}{\lfloor an\rfloor-1}}{\binom{n}{\lfloor an\rfloor}} \\
                            &=\sum_{j=0}^{\ell}Q(\lfloor an\rfloor,n-j,\ell -j)\cdot \frac{\binom{n-j-1}{\lfloor an\rfloor-1}}{\binom{n}{\lfloor an\rfloor}} \\
        \end{aligned}
    \end{equation*}
    where in the last equality, we re-indexed, setting $j=n-m$. 
    The interval of summation does not depend on $n$, so we may consider the asymptotics of the summand. For the remainder of the proof, we will use the falling factorial estimate 
    \[(n)_j\sim n^j\;\;\;\text{as}\;\;\; n\to\infty\]
    freely. We have
    \[\frac{\binom{n-j-1}{\lfloor an\rfloor-1}}{\binom{n}{\lfloor an\rfloor}}=\frac{\lfloor an\rfloor}{n}\cdot\frac{(n-\lfloor an\rfloor)_j}{(n)_j}\sim a\cdot(n(1-a))^j/(na)^j\to a(1-a)^j.\]
    On the other hand, using Lemma $\ref{formula for qkms}$, we have
    \begin{equation*}
        \begin{aligned}
            Q(\lfloor an\rfloor,n-j,\ell -j) & \sim \frac{1}{(\ell-j)!}\sum_{t=\ell-j}^{an}\frac{(-1)^{t-(\ell-j)}(an)^t}{(t-(\ell -j))!(n-j)^t} \\
            & \sim \frac{1}{(\ell-j)!}\sum_{t=\ell-j}^{an}\frac{(-1)^{t-(\ell-j)}a^t}{(t-(\ell -j))!} \\
            & =\frac{a^{\ell -j}}{(\ell-j)!}\sum_{t=0}^{an-(\ell -j)}\frac{(-1)^t a^t}{t!} \\
            & \to \frac{e^{-a}\cdot a^{\ell -j}}{(\ell -j)!}.
        \end{aligned}
    \end{equation*}
    where the limit is taken as $n\to\infty$. We conclude that
    \begin{equation}
    \label{expression for poisson-geometric limit deterministic}
        \lim_{n\to\infty} \Prob(L_{a,n}=\ell)=\sum_{j=0}^{\ell}\frac{a(1-a)^j e^{-a}a^{\ell -j}}{(\ell -j)!}=\Prob(X+Y=\ell)
    \end{equation}
    where the probability mass function of $X+Y$ is from Equation $\eqref{pmf of x+y}$. Hence $L_{a,n}\Rightarrow_d X+Y$. The same argument is uniform for any constant $a$ in compact subsets of $(0,1]$. Therefore, if $a_n\to a$ and the sequence $(a_n)$ is bounded away from zero, then $L_{a_n,n}$ has the same distributional limit as $L_{a,n}$.
\end{proof}
\begin{figure}   
\centering

\includegraphics[height=0.25\textheight]{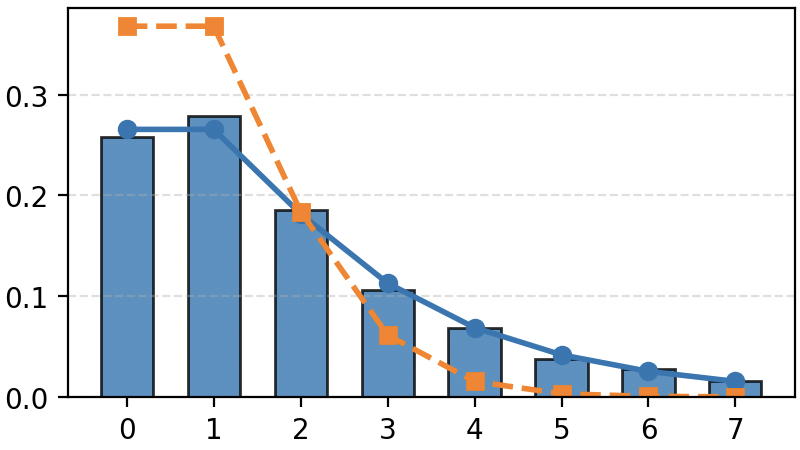}

\vspace{0.3cm}

\includegraphics[height=0.25\textheight]{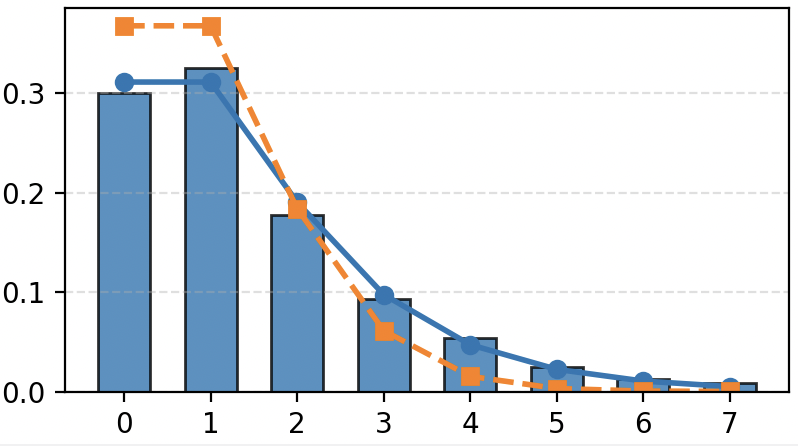}

\vspace{0.3cm}

\includegraphics[height=0.25\textheight]{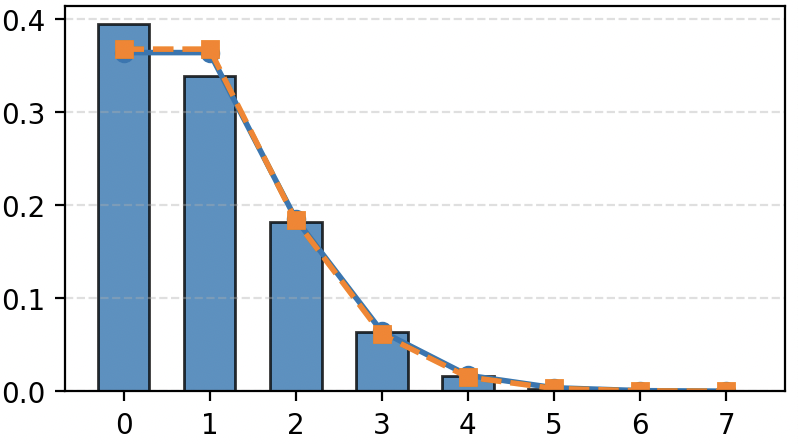}

\caption{Histograms from $2000$ trials of the number of fixed points of $n=10000$-card decks after iterated random-to-top shuffles, normalized to form probability distributions. The overlays are linearly interpolated versions of the predicted Poisson--geometric convolutions from Theorem~\ref{fixed points main} in blue and a $\Poi(1)$ distribution in orange. The top histogram corresponds to $r=5000$ shuffles, the middle to $r=10000$ shuffles, and the bottom to $r=20000$ shuffles.}
\label{fig: fixed points}
\end{figure}
We now turn to the corresponding statistic indexed by $K_{cn,n}$. The next lemma helps control the fluctuation between the deterministic and random-indexed statistics.
\begin{lemma}
\label{bounds on max and min}
Let $X$ be a random variable such that $X\geq 0$ a.s. and $\Ex [X]=\mu<\infty$. Then, we have
\[\Ex\left[\frac{\min(\mu,X)}{\max(\mu,X)}\right]\geq 1-\frac{\Ex|X-\mu|}{\mu}.\]
\end{lemma}
\begin{proof}
If $X\leq \mu$, then $\min(\mu,X)/\max(\mu,X)=X/\mu$. Otherwise, $X\geq \mu$, and then the ratio equals $\mu/X$. In either case,
\[
1-\frac{\min(\mu,X)}{\max(\mu,X)}\leq \frac{|X-\mu|}{\mu}.
\]
Taking expectations yields the claim.
\end{proof}

We are now ready to prove Theorem $\ref{fixed points main}$.

\begin{proof}[Proof of Theorem $\ref{fixed points main}$]
    Let
    \[
    a_n:=\frac{1}{n}\Ex[K_{cn,n}].
    \]
    By Proposition~\ref{facts about krn}, we have $a_n\to 1-e^{-c}$ and, moreover, $a_n\ge 1-e^{-c}$ for all $n$. By Proposition~\ref{prop convergence of deterministic}, the deterministic model $L_{a_n,n}$ from Equation $\eqref{convergence of deterministic model}$ converges in distribution to $X+Y$, where $X$ and $Y$ are independent, $X\sim \Poi(1-e^{-c})$, and $Y$ is geometric with $P(Y=k)=(1-e^{-c})e^{-ck}$.

It remains to compare $L_{a_n,n}$ with the random-indexed model
    \[
    L_{(n^{-1}K_{cn,n}),n}
    =
    n-\max_{1\le i\le K_{cn,n}}\pi(i)+\sum_{i=1}^{K_{cn,n}}\mathbf{1}(\pi(i)=i).
    \]
    By Theorem $\ref{decomposition of fixed points}$, we have that $L_{(n^{-1}K_{cn,n}),n}=_d F_{cn,n}$. Let 
    \[M_{K_{cn,n}}=\max_{1\leq i\leq K_{cn,n}} \pi(i)\;\;\;\;\text{and}\;\;\;\;M_{a_n n}=\max_{1\leq i\leq \lfloor a_n n\rfloor } \pi(i),\]
    and set 
    \[H_{K_{cn,n}}=\sum_{i=1}^{K_{cn,n}} \1(\pi(i)=i)\;\;\;\;\text{and}\;\;\;\;H_{a_n n}=\sum_{i=1}^{\lfloor a_n n\rfloor} \1(\pi(i)=i).\]
    We have
    \[L_{a_n,n}-L_{(n^{-1}K_{cn,n}),n}=H_{a_n n}-H_{K_{cn,n}}+M_{K_{cn,n}}-M_{a_nn},\]
    so it suffices to show that $H_{K_{cn,n}
    }-H_{a_nn}$ and $M_{K_{cn,n}}-M_{a_nn}$ both converge to zero in probability. We first control the fixed-point term.
    By the law of iterated expectations, Proposition $\ref{facts about krn}$, and Cauchy-Schwarz, we have
    \begin{equation}
    \label{cauchy schwarz and markov}   
        \Ex|H_{a_n n}-H_{K_{cn,n}}| = \frac{\Ex|K_{cn,n}-\Ex[K_{cn,n}]|}{n} \leq \frac{\sqrt{\Var K_{cn,n}}}{n}\to 0 
    \end{equation}
    as $n\to\infty$. Therefore, $H_{a_n n}-H_{K_{cn,n}}\to 0$ in $L^1$, and hence in probability.
    
    The next step is to control the maximum term $M_{K_{cn,n}}-M_{a_nn}$. Let $\alpha_n=\min(\Ex[K_{cn,n}],$ $K_{cn,n})$, and let $\beta_n=\max(\Ex[K_{cn,n}],$ $K_{cn,n})$. We see that $M_{K_{cn,n}}-M_{a_nn}\neq 0$ if and only if $\max_{i\in [\alpha_n]}(\pi(i))$ $<\max_{\alpha_n +1\leq i\leq\beta_n} (\pi(i))$. We have, for all $\varepsilon\in (0,1)$,
    \begin{equation*}
        \begin{aligned}
            \Prob(|M_{K_{cn,n}}-M_{a_nn}|>\varepsilon) & =\Prob\left(\max_{i\in [\alpha_n]}(\pi(i))<\max_{\alpha_n +1\leq i\leq\beta_n} (\pi(i))\right) \\
            & =\Ex\left[\Prob\left(\max_{i\in [\alpha_n]}(\pi(i))<\max_{\alpha_n +1\leq i\leq\beta_n} (\pi(i))\right)\mid K_{r,n}\right] \\
            & =\Ex\left[(\beta_n-\alpha_n)/\beta_n\right] \\
            & =1-\Ex\left[\alpha_n/\beta_n\right]
    \leq \frac{\Ex|K_{r,n}-\Ex[K_{r,n}]|}{\Ex[K_{r,n}]}
        \end{aligned}
    \end{equation*}
    where we used Lemma $\ref{bounds on max and min}$ in the last inequality. The argument above also works for $\varepsilon\geq 1$ by monotonicity. Using the Cauchy-Schwarz bound from Equation $\eqref{cauchy schwarz and markov}$ together with Proposition $\ref{facts about krn}$, we conclude that $M_{K_{cn,n}}-M_{a_nn}\to 0$ in probability. Combining this with the convergence in probability of the fixed-point term $H_{a_n n}-H_{K_{cn,n}}$, we deduce that $L_{a_n,n}-L_{(n^{-1}K_{cn,n}),n}\to 0$ in probability. By Lemma $\ref{slutsky ii}$, we conclude Theorem $\ref{fixed points main}$ in the critical regime. 
    
    We now treat the mixed regime. Let $\pi$ be a uniformly random permutation of $[n]$, and define
    \[M_{K_{r,n}}=\max_{1\leq i\leq K_{r,n}}\pi(i)\;\;\;\text{and}\;\;\; H_{K_{r,n}}=\sum_{i=1}^{K_{r,n}} 1(\pi(i)=i).\]
    Observe that, by Proposition $\ref{facts about krn}$,
    \begin{equation*}
        \begin{aligned}
            \Prob\left(M_{K_{r,n}}=n\right)& =\sum_{k=1}^n\Prob(M_{K_{r,n}}=n\, |\, K_{r,n}=k)\Prob(K_{r,n}=k) \\
            & =\sum_{k=1}^n\frac{k}{n}\Prob(K_{r,n}=k) \\
            &=\frac{\Ex[K_{r,n}]}{n}=1-\left(1-\frac{1}{n}\right)^r.
        \end{aligned}
    \end{equation*}
    This probability tends to $1$ when $r=\omega(n)$, so we conclude that $n-M_{K_{r,n}}\to 0$ in probability in this case. By Theorem $\ref{slutsky}$, $F_{r,n}$ and $H_{K_{r,n}}$ have the same limiting distribution when $r=\omega(n)$. We can write
    \[H_{K_{r,n}}=\sum_{i=1}^n \1(\pi(i)=i) -\underbrace{\sum_{i=K_{r,n}+1}^n \1(\pi(i)=i)}_{:=\,G_{r,n}}.\]
    By the law of iterated expectations, we have
    \begin{equation*}
            \Ex[G_{r,n}]= \frac{\Ex[n-K_{r,n}]}{n} =(1-1/n)^r\to 0
    \end{equation*}
    as $n\to\infty$ when $r=\omega(n)$. This implies that $G_{r,n}\to 0$ in $L^1$ and hence in probability in this case. Applying Theorem $\ref{slutsky}$ again, we see that $F_{r,n}$ and the number of fixed points of a uniformly random permutation share the same limiting distribution when $r=\omega(n)$. The convergence in distribution of $F_{r,n}$ to a $\Poi(1)$ random variable when $r=\omega(n)$ hence follows from Proposition $\ref{uniform case convergences}$. This concludes the proof of Theorem $\ref{fixed points main}$.
\end{proof}

\section{Descents}
\label{Section 6}
In this section, we investigate the limiting distribution of the number of descents after iterated top-to-random shuffles and develop the tools needed to prove Theorem $\ref{descents main}$. We conclude the section with its proof.

We begin by analyzing a centered version of the decomposition in Theorem $\ref{decomposition of descents}$, where the random index is replaced by a deterministic sequence. In this setting, we establish convergence of the normalized descent statistic.

\begin{proposition}
\label{convergence of centered descents}
Let $a_n\in (0,1)$ be a deterministic sequence such that $a_n\to a\in(0,1)$. Let $(U_i)_{i\geq1}^n$ be i.i.d. continuous uniform random variables on $(0,1)$. Set
\begin{equation}
\label{deterministic descents}
    f(a_n)=\sum_{i=1}^{\lfloor a_n n\rfloor -1} \left(\1(U_i>U_{i+1})-\frac{1}{2}\right).
\end{equation}
Then, we have
\[\frac{f(a_n)}{\sqrt{n}}\Rightarrow_d\mathcal{N}(0,a/12).\]
\end{proposition}
\begin{proof}
Let 
$m_n = \lfloor a_n n \rfloor$, and set $W_i=\1(U_i>U_{i+1})-1/2$. Then, by a covariance expansion,
\[
\Var(f(a_n))=\sum_{i=1}^{m_n-1} \Var(W_i)+2\left(\sum_{1 \le i < j \le m_n-1} \Cov(W_i, W_j)\right).
\]
Since the random variables $W_i$ are one-dependent, $\Cov(W_i,W_j)=0$ whenever $|i-j|>1$, and hence
\[
\Var(f(a_n))
=
(m_n-1)\Var(W_1)
+
2(m_n-2)\Cov(W_1, W_2).
\]
Now,
\[
\Var(W_1)=\frac{1}{4} \;\;\;\text{and}
\;\;\;
\Cov(W_1,W_2)=\Prob(U_1>U_2>U_3)-\frac{1}{4}
=\frac{1}{6}-\frac{1}{4}=-\frac{1}{12}.
\]
Therefore,
\[
\Var(f(a_n))
=
(m_n-1)\frac{1}{4}
+
2(m_n-2)\left(-\frac{1}{12}\right)
=
\frac{m_n+1}{12}.
\]
Consequently,
\[
\Var\left(\frac{f(a_n)}{\sqrt{n}}\right)
=
\frac{m_n+1}{12n}
\to \frac{a}{12},
\]
since $m_n = \lfloor a_n n \rfloor$ and $a_n \to a$.
By the central limit theorem for $m$-dependent random variables as in \cite{m-dependentclt-hoeffding}, we have that the convergence in distribution of $n^{-1/2}f(a_n)$ holds.
\end{proof}

\begin{figure}[htbp]
\centering

\includegraphics[height=0.28\textheight,keepaspectratio]{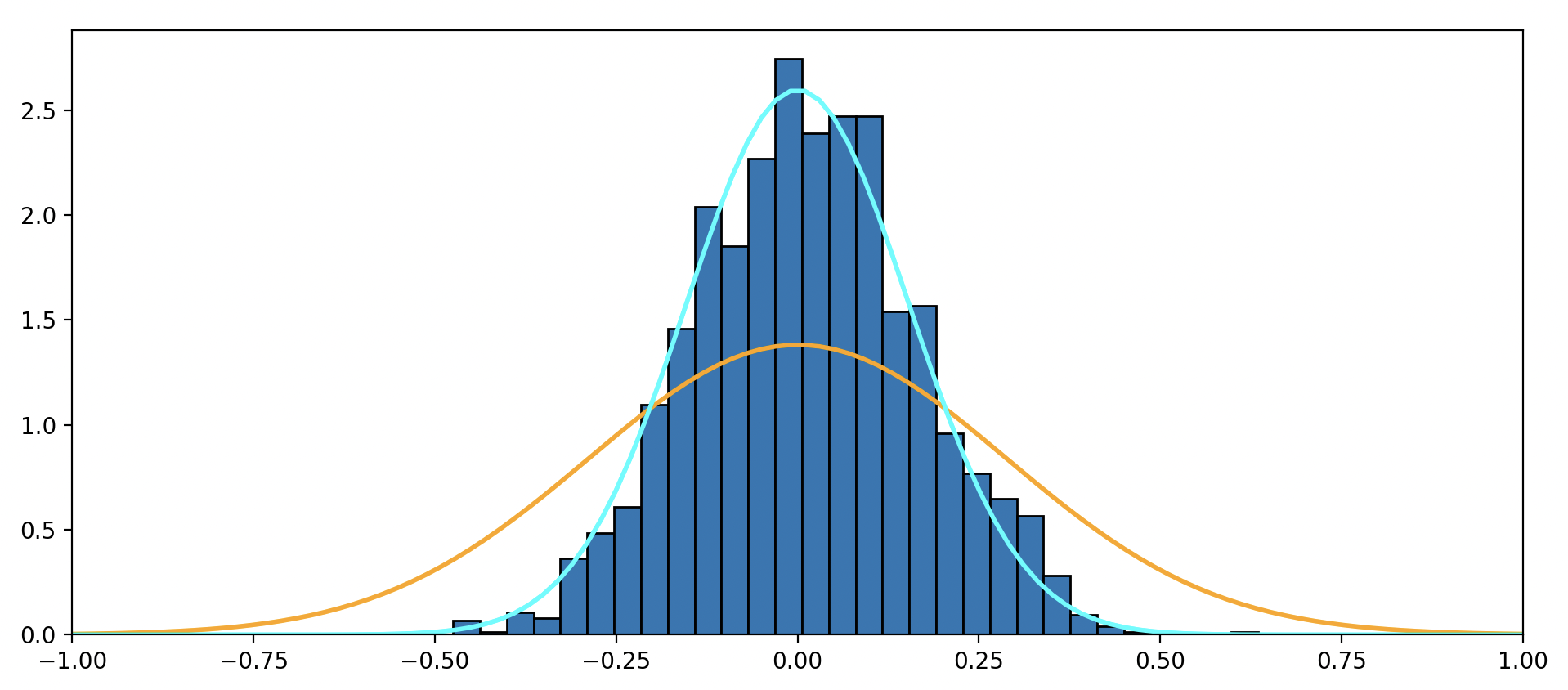}

\vspace{0.3cm}

\includegraphics[height=0.28\textheight,keepaspectratio,width=0.85\textwidth]{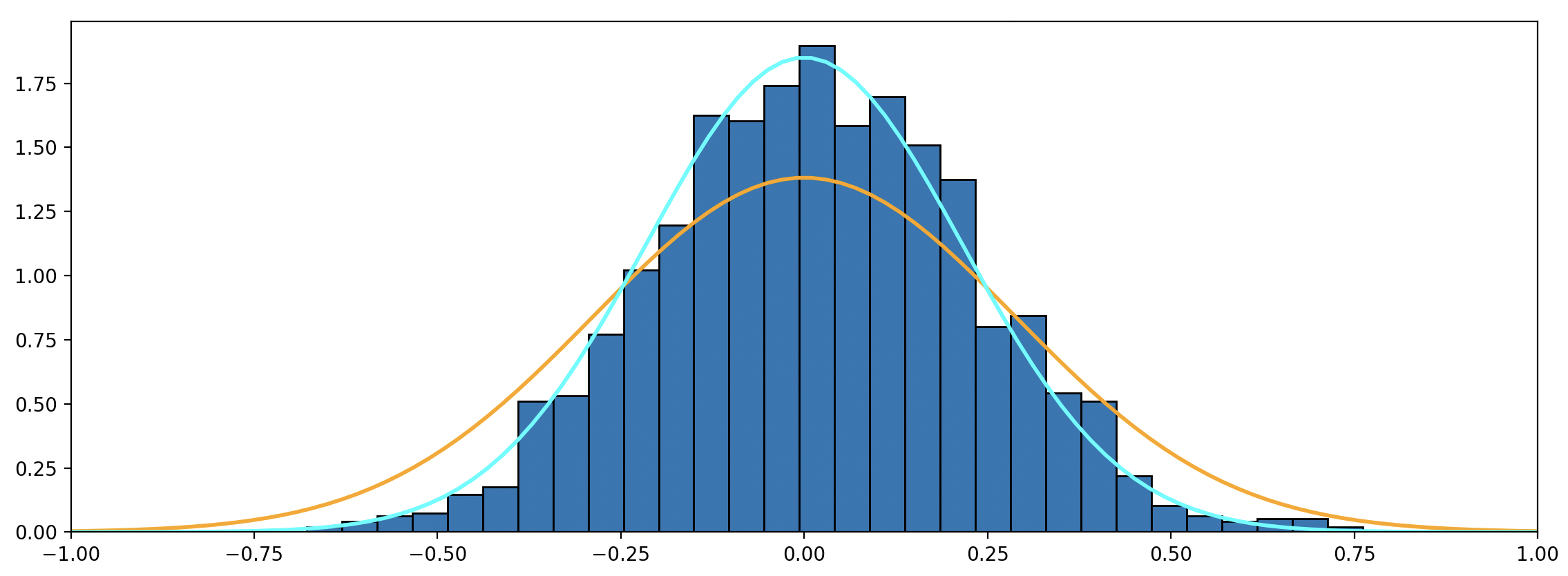}

\vspace{0.3cm}

\includegraphics[height=0.28\textheight,keepaspectratio]{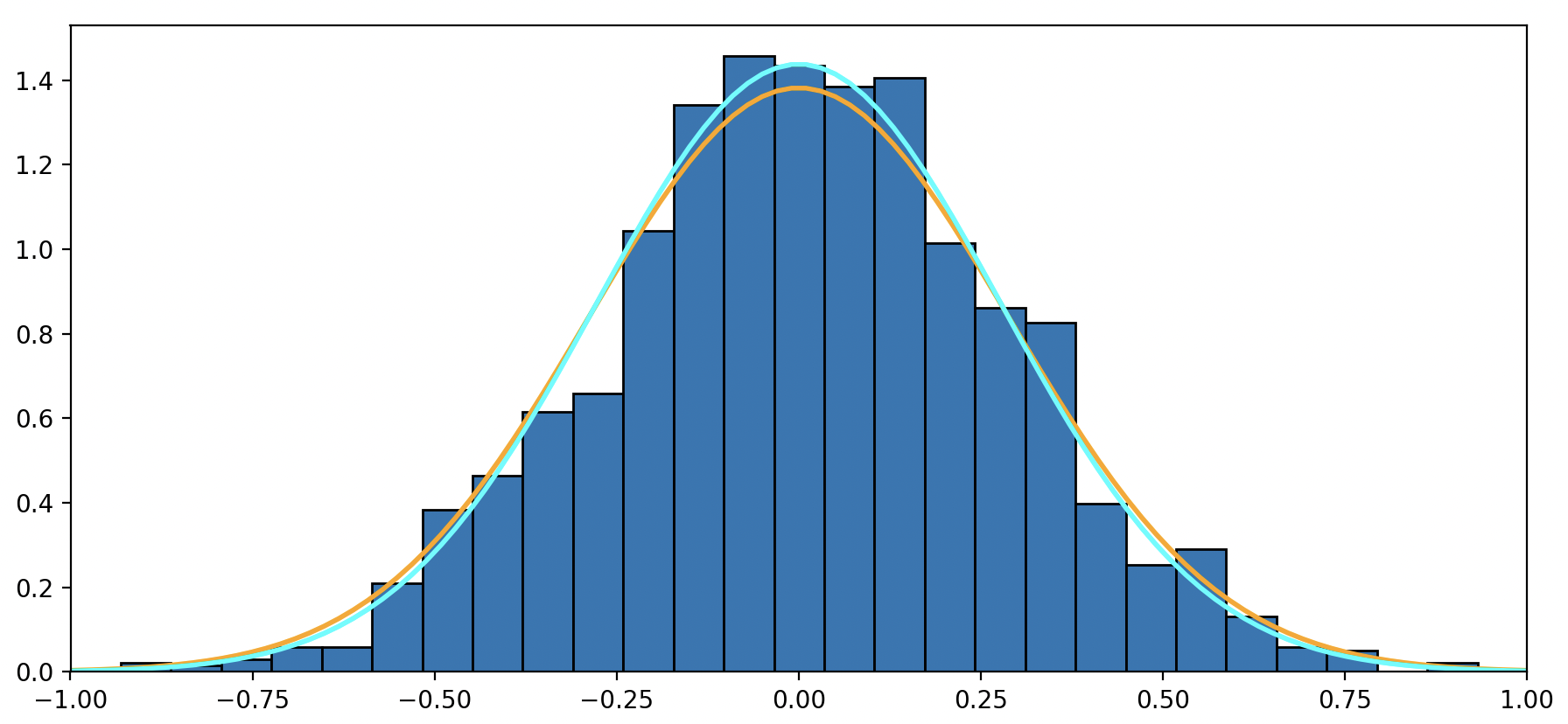}

\caption{Histograms based on $2000$ trials of the number of descents for $n=10000$ under iterated random-to-top shuffles, centered by $n(1-e^{-r/n})$ and scaled by $n^{-1/2}$, and normalized to form probability distributions. Top: $r=2500$. Middle: $r=5000$. Bottom: $r=10000$. The cyan curves show the predicted limiting densities in the $r=cn$ regime from Theorem~\ref{descents main}, while the orange curves correspond to the $\mathcal{N}(0,1/12)$ density.}
\label{fig:descents-comparison}

\end{figure}

Now consider comparing the statistic $f(a_n)$ with the corresponding randomly centered, randomly indexed statistic. We can show that in a distributional limit with $n^{-1/2}$ scaling, under proper assumptions on the indexing random variables, the randomly indexed and centered statistic can be replaced by $f(a_n)$. This leads us to the following general result.
\begin{theorem}
\label{descents technical}
    Let $K_n$ be a sequence of random variables satisfying the following assumptions.
    \begin{enumerate}
        \item $K_n/n\to a\in (0,1)$ in probability.
        \item $n^{-1/2}(K_n-\Ex K_n)\Rightarrow_d\mathcal{N}(0,\tau^2)$ for some fixed $\tau\in[0,\infty)$.
        \item $\Var K_n=O(n)$ and $\Ex[K_n]=an+o(\sqrt{n})$.
    \end{enumerate} 
    Let $U_i$ be iid continuous uniform on $(0,1)$. Set $T_n=\sum_{i=1}^{\lfloor K_n\rfloor -1}\, \1(U_i>U_{i+1})$. We have
    \[\frac{T_n-\Ex [T_n]}{\sqrt{n}}\Rightarrow_d\,\mathcal{N}\left(0,\frac{a}{12}+\frac{\tau^2}{4}\right).\]
\end{theorem}
\begin{proof}
    We rewrite 
    \begin{equation}
    \label{an and bn descents}
        T_n-\Ex[T_n]=\underbrace{T_n-\Ex[T_n\mid K_n]}_{:=\,A_n}+\underbrace{\Ex[T_n\mid K_n]-\Ex[T_n]}_{:=\, B_n}.
    \end{equation}
    Let $a_n=n^{-1}\cdot\Ex[K_n]$, and let $f(a_n)$ be as in Equation $\ref{deterministic descents}$. We will show that $n^{-1/2}\cdot|A_n-f(a_n)|\to 0$ in probability. Let $W_i=\1(U_i>U_{i+1})-1/2$. Note that 
    \[A_n=\sum_{i=1}^{\lfloor K_n\rfloor -1} W_i.\] Let $\alpha_n=\min(a_n n,K_n)$, and set $\beta_n=\max(a_n n,K_n)$. We have

    \begin{equation*}
        \begin{aligned}
        \Ex[(A_n-f(a_n))^2\mid K_n]&=\sum_{i=\alpha_n-1}^{\beta_n-1}\Var W_i+2\cdot\sum_{i=\alpha_n-1}^{\beta_n-2}\text{Cov}(W_i,W_{i+1})\\
        &=\frac{\beta_n-\alpha_n+1}{12}=\frac{|K_n-a_n n|+1}{12}.
        \end{aligned}
    \end{equation*}
    Taking expectations and applying Cauchy Schwarz and Assumption (3), we have that $\Ex[(A_n-f(a_n))^2]= o(n)$,
    so that $n^{-1/2}\cdot|A_n-f(a_n)|\to 0$ in $L^2$ and hence in probability.

    We turn our attention to the random variable $B_n$ from Equation $\eqref{an and bn descents}$. By direct computation and Assumption (3), we have
    \begin{equation}
    \label{formula for bn}
        B_n=\frac{K_n-an}{2}+o(\sqrt{n}).
    \end{equation}
    By Proposition $\ref{convergence of centered descents}$ and Assumption (1), we have that 
    \[\frac{f(a_n)}{\sqrt{n}}\Rightarrow_d\,\mathcal{N}(0,a/12).\]
    Then, by Lemma $\ref{slutsky ii}$ applied to the pair of random variables $A_n$ and $f(a_n)$, along with Equation $\ref{formula for bn}$ and Assumption (2), we have
    \begin{equation*}
        \begin{aligned}
        \frac{T_n-\Ex [T_n]}{\sqrt{n}}&=\frac{A_n}{\sqrt{n}}+\frac{B_n}{\sqrt{n}}\\
        &=\frac{f(a_n)}{\sqrt{n}}+\frac{K_n-an}{2\sqrt{n}}+o_p(1)\\
        &\Rightarrow_d\,\mathcal{N}\left(0,\frac{a}{12}+\frac{\tau^2}{4}\right).
        \end{aligned}
    \end{equation*}
    The last convergence in distribution follows from the independence of $f(a_n)$ and $K_n$ since $K_n$ and the $(U_i)_{i\geq 1}^n$ are independent, along with Proposition $\ref{convergence of centered descents}$ and Assumption (2).
\end{proof}
We can now deduce the limiting distribution of the number of descents $D_{r,n}$ in the critical regime as a special case of Theorem $\ref{descents technical}$. On the other hand, the result in the mixed regime will follow from manipulating the decomposition in Theorem $\ref{decomposition of descents}$ and applying Theorem $\ref{slutsky}$.
\begin{proof}[Proof of Theorem $\ref{descents main}$]
We begin by proving the $r=cn$ case, i.e. the critical regime. This follows from Theorem $\ref{descents technical}$ applied to the occupancy random variables $K_n:=K_{cn,n}$. The variance $a/12 +\tau^2/4$ from Theorem $\ref{descents technical}$ is computed with $a=1-e^{-c}$ and $\tau^2=e^{-c}-(1+c)e^{-2c}$.
By Proposition $\ref{facts about krn}$, the random variables $K_{cn,n}$ satisfy Assumptions (1)--(3) in Proposition $\ref{descents technical}$. 

We now derive the limiting distribution in the mixed regime. From the equality in distribution in Theorem $\ref{decomposition of descents}$, we can rewrite
\[D_{r,n}=_d \,O_p(1)+\sum_{i=1}^{n-1}\1(U_i>U_{i+1})-\underbrace{\sum_{i=K_{r,n}}^{n-1}\1(U_i>U_{i+1})}_{:=\,G_{r,n}}\]
where the $U_i$ are i.i.d. continuous uniform$(0,1)$ random variables. We will find conditions on $r$ that ensure $n^{-1/2}\cdot G_{r,n}\to 0$ in probability. Using the law of iterated expectations and Proposition $\ref{facts about krn}$, we find that
\begin{equation*}
    \begin{aligned}
        n^{-1/2}\cdot \Ex[G_{r,n}]&=\frac{n-\Ex[K_{r,n}]}{2\sqrt{n}}+o(1)\\
        &=\frac{\sqrt{n}\cdot(1-1/n)^r}{2} + o(1)\\
        &\leq \frac{\sqrt{n}\cdot\exp(-r/n)}{2}+o(1).
    \end{aligned}
\end{equation*}
Hence, $G_{r,n}\to 0$ in $L^1$, and therefore in probability, when there exists a sequence $c_n\to\infty$ such that $r\gtrsim (n\log n)/2 +c_n n$. We conclude by Theorem $\ref{slutsky}$ that $n^{-1/2}\cdot D_{r,n}$ has the same limiting distribution as the number of descents of a uniformly random permutation in this case. This concludes the proof of Theorem $\ref{descents main}$.
\end{proof}

\section{Inversions}
\label{Section 7}
In this section, we develop tools to prove convergence in distribution of the number of inversions under random-to-top shuffles. The section concludes with the proof of Theorem $\ref{inversions main}$. We will use similar methods as in Section $\ref{Section 6}$ to prove the asymptotic normality of inversions. We begin by finding the expected number of inversions in a combinatorial manner.
\begin{proposition}
    Let $I_{r,n}$ be the number of inversions after $r$ iterated random-to-top shuffles of an $n$-card deck. We have
    \[\Ex [I_{r,n}]=\frac{\binom{n}{2}}{2}\left(1-\left(\frac{n-2}{n}\right)^r\,\right).\]
\end{proposition}
\begin{proof}
   Let $r$ balls be thrown into $n$ labeled boxes uniformly at random. A pair $(i,j)$ is an inversion after iterated random-to-top shuffles if and only if card $j$ is moved to the top sometime after the last time card $i$ is. This corresponds to box $j$ receiving a ball sometime after the last time box $i$ receives a ball. Let $T_i$ be the last throw that box $i$ receives a ball, and set $T_i=0$ if box $i$ never receives a ball. By the reasoning above, we have
    \[I_{r,n}=_d\sum_{i<j} \1(T_j>T_i).\]
    Summing over all possible indicators, we have
    \[\binom{n}{2}=\sum_{i<j} \1(T_j>T_i)+\sum_{i<j} \1(T_j<T_i)+\sum_{i<j} \1(T_j=T_i=0).\]
    Note that $\Prob(T_i=T_j=0)=((n-2)/n)^r$. Taking expectations of both sides and observing that $\mathbb{P}(T_i>T_j)=\mathbb{P}(T_i<T_j)$ for any $i\neq j$ allows us to solve for $\Ex[I_{r,n}]$.
\end{proof}
We note that the expected number of inversions is found in Section 8.4 of Diaconis and Fulman \cite{MathofShufflingCards} using spectral theory, but it appears that our combinatorial proof is new. 

\begin{figure}[htbp]
\centering

\includegraphics[height=0.25\textheight,keepaspectratio]{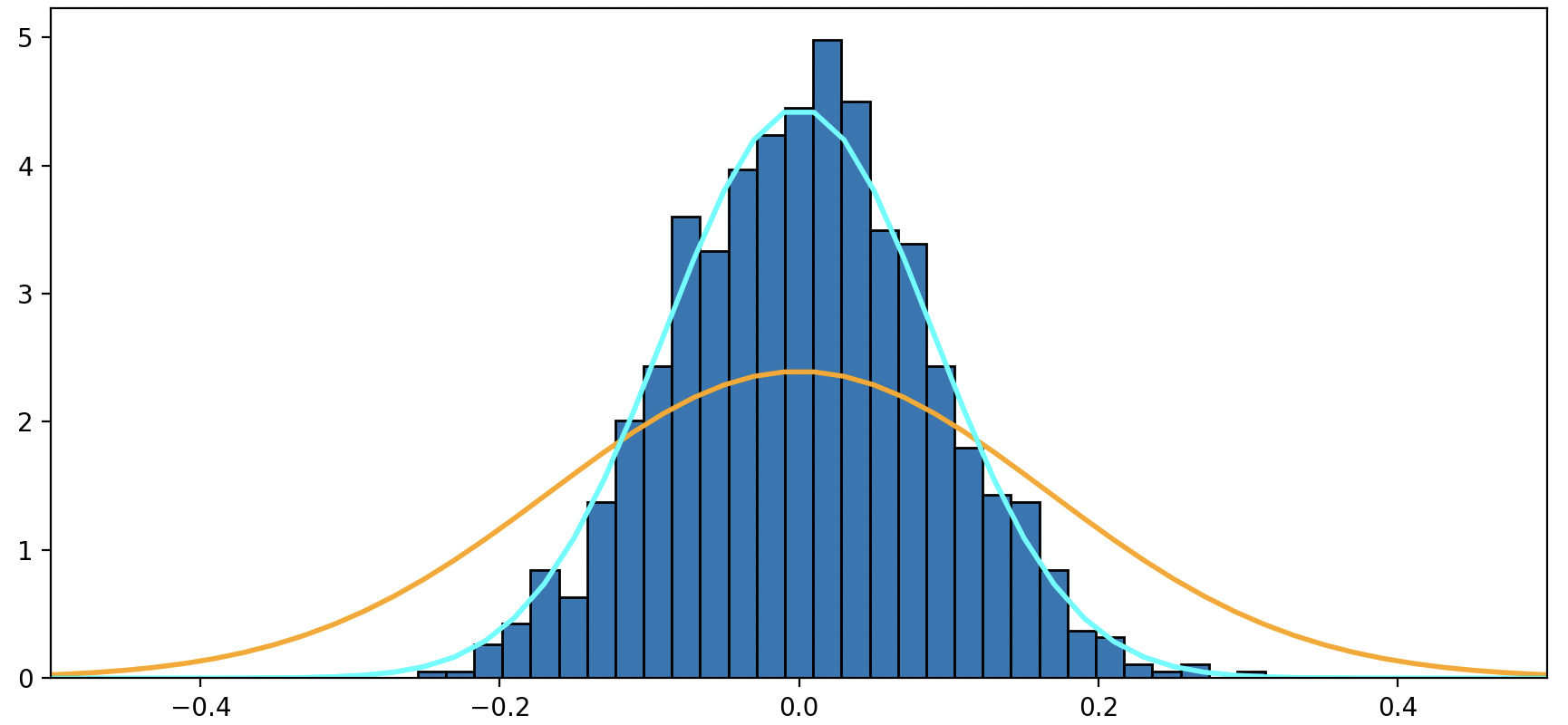}

\vspace{0.2cm}

\includegraphics[height=0.25\textheight,keepaspectratio,width=0.85\textwidth]{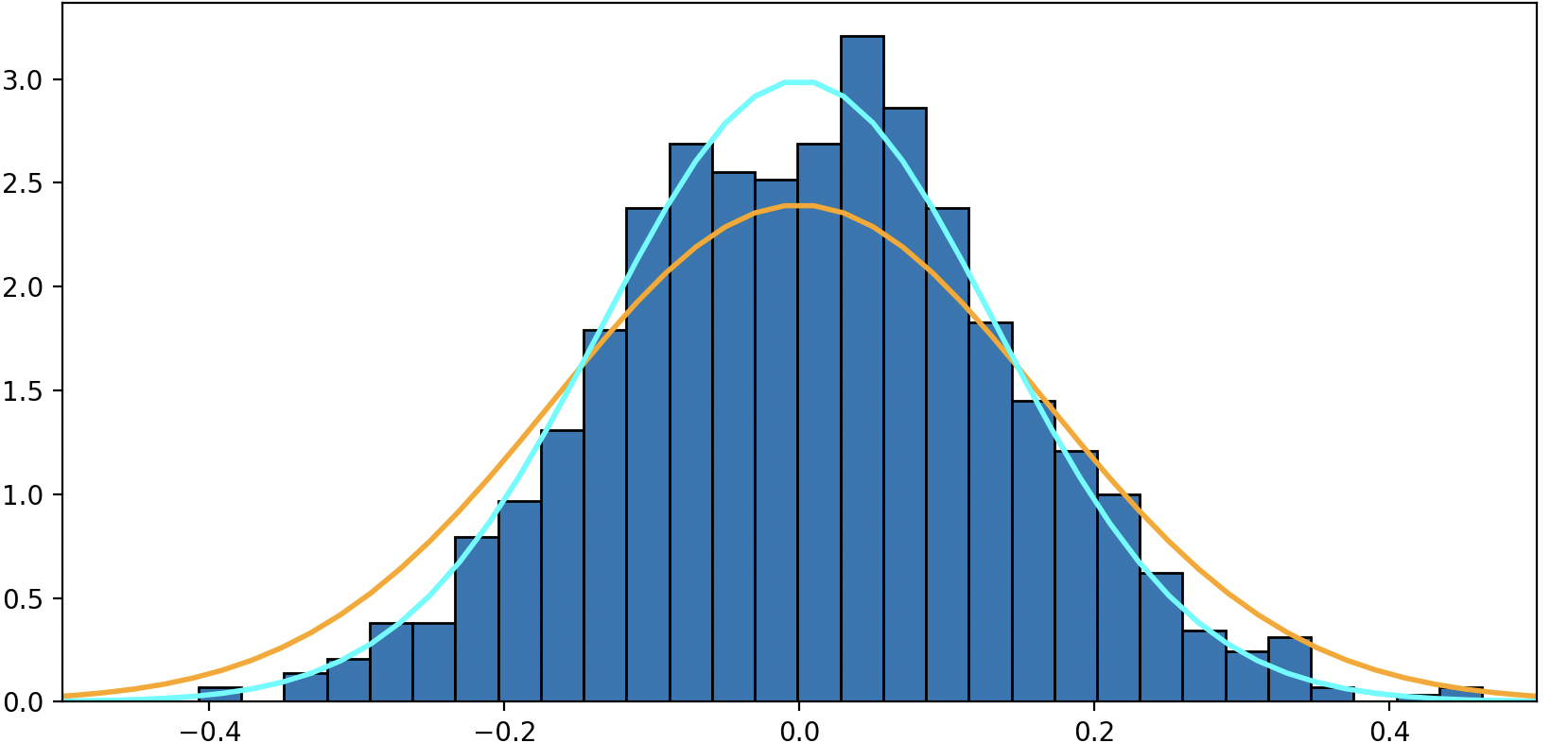}

\vspace{0.2cm}

\includegraphics[height=0.25\textheight,keepaspectratio]{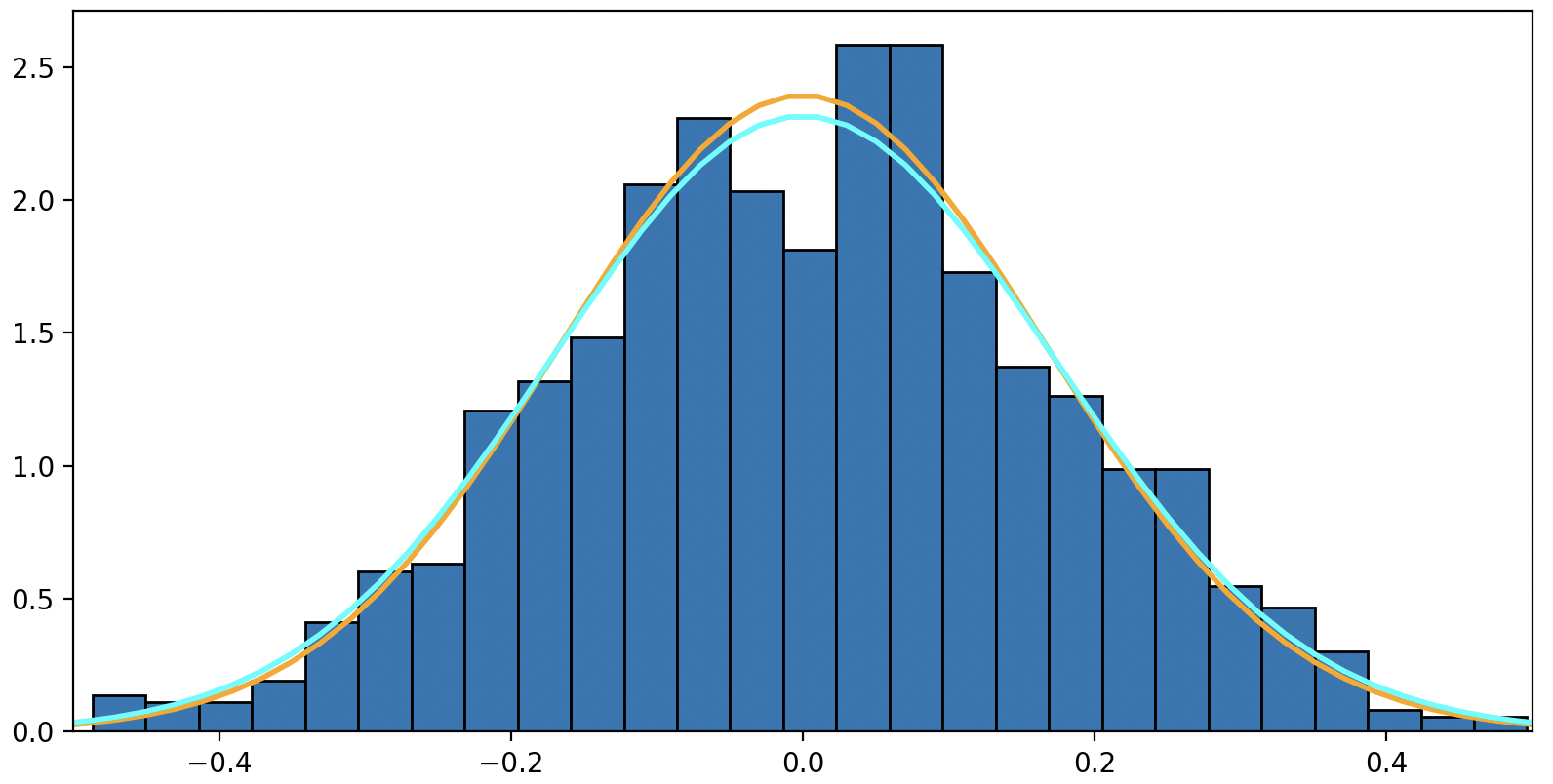}

\caption{Histograms of $1000$ trials of the number of inversions of $n=1000$ card decks after iterated random-to-top shuffles centered by $n(1-e^{-2r/n})$ and scaled by $n^{-3/2}$, normalized to form probability distributions. Top: $r=100$ shuffles. Middle: $r=250$ shuffles. Bottom: $r=1000$ shuffles. The cyan curves are the predicted $r=cn$ densities from Theorem $\ref{inversions main}$. The orange curves are $\mathcal{N}(0,1/36)$ densities.}
\label{fig:inversions-comparison}
\end{figure}

Analogously to our consideration of descents in the $r=cn$ case, consider the decomposition of inversions from Theorem $\ref{decomposition of inversions}$, where we have
\[I_{cn,n}=_d\sum_{i=1}^{K_{cn,n}} R_i.\]
Here, the $R_i$ are independent discrete uniform random variables supported on $\{0,1,\ldots,$ $n-i\}$. If we replace $K_{cn,n}$ with a deterministic index $\lfloor a_n n\rfloor $ with $a_n\to a\in(0,1)$, one would expect that by the independence of the $R_i$, we would get asymptotic normality. Indeed, this is the case.

\begin{proposition}
\label{deterministic convergence inversions}
    Let $R_i$ be independent discrete uniform random variables on $\{0,1,$ $\ldots,n-i\}$. Let $a_n$ be a deterministic sequence with $a_n\to a\in(0,1)$. Set 
    \[f(a_n)=\sum_{i=1}^{\lfloor a_n n\rfloor} (R_i-(n-i)/2).\] 
    We have
    \[\frac{f(a_n)}{n^{3/2}}\Rightarrow_d\,\mathcal{N}\left(0,\frac{1-(1-a)^3}{36}\right).\]
\end{proposition}
\begin{proof}
    Computing the variance, we have
    \begin{equation*}
        \begin{aligned}
        \Var f(a_n) &=\sum_{i=1}^{\lfloor a_n n\rfloor}\Var R_i\\
        &=\sum_{i=1}^{\lfloor a_n n\rfloor}\frac{(n-i)(n-i+2)}{12} \\
        &\sim\frac{1}{12}\int_{(1-a_n)n}^n x^2\,dx\sim n^3\cdot\frac{1-(1-a)^3}{36}
        \end{aligned}
    \end{equation*}
    Let $W_i=(R_i-(n-i)/2)$. We note that $|W_i|\leq n$ a.s. for every $i$, so the Lindeberg condition is satisfied as $|W_i|/\sqrt{\Var{f(a_n)}}\to 0$. The central limit theorem for triangular arrays as in Chapter 3 of \cite{durrett2019probability} implies the result.
\end{proof}

Using Proposition $\ref{deterministic convergence inversions}$, we are able to prove a more general result which works, under certain assumptions, for a general indexing random variable. Then, we will specialize our result to the case where the statistic is indexed by the random variable $K_{cn,n}$.

\begin{theorem}
\label{general inversions}
    Let $R_i$ be the independent discrete uniform random variables from Proposition $\ref{deterministic convergence inversions}$. Let $K_n$ be a sequence of random variables which are independent of the $R_i$ and satisfy the following assumptions.
    \begin{enumerate}
        \item $K_n/n\to a\in (0,1)$ in probability.
        \item $n^{-1/2}(K_n-\Ex[K_n])\Rightarrow_d\mathcal{N}(0,\tau^2)$ for some fixed $\tau^2\in[0,\infty)$.
        \item $\Var K_n=O(n)$ and $\Ex [K_n]=an+o(\sqrt{n})$.
    \end{enumerate}
    Then, letting $T_n=\sum_{i=1}^{\lfloor K_n\rfloor } R_i$, we have
    \[\frac{T_n-\Ex T_n}{n^{3/2}}\Rightarrow_d\,\mathcal{N}(0,\sigma^2(a,\tau^2))\]
    where 
        \[\sigma^2(a,\tau^2)=\frac{1-(1-a)^3}{36}+
    \frac{(1-a)^2}{4}\,\tau^2.
    \]
\end{theorem}
\begin{proof}
    The proof proceeds in a similar manner to the proof of Theorem $\ref{descents technical}$. Let $a_n=n^{-1} \cdot \Ex K_n$ so that $a_n\to a$ by Assumption (3). We begin by writing
    \begin{equation}
    \label{an and bn inversions}
    T_n-\Ex [T_n]=\underbrace{T_n-\Ex [T_n\mid K_n]}_{:=\,A_n}+\underbrace{\Ex [T_n\mid K_n]-\Ex [T_n]}_{:=\, B_n}
    \end{equation}
    Let $f(a_n)$ be as in Proposition $\ref{deterministic convergence inversions}$. We will show that $n^{-3/2}\cdot |A_n-f(a_n)|\to 0$ in probability. Let $\alpha_n=\min(a_n n,K_n)$, and set $\beta_n=\max(a_n n,K_n)$. We have 
    \[\Ex [(A_n-f(a_n))^2\mid K_n]=\sum_{i=\alpha_n+1}^{\beta_n}\mathrm{Var}(R_i)
    \leq \frac{n^2}{12}\,|K_n-a_nn|\]
    almost surely. Taking expectations and applying both the Cauchy-Schwarz inequality and Assumption (3) gives
    \[n^{-3}\cdot \Ex [(A_n-f(a_n))^2]\leq (12n)^{-1}\sqrt{\Ex (K_n-a_n n)^2}
    =O(n^{-1/2}).\]
    We conclude that $n^{-3/2}\cdot |A_n-f(a_n)|\to 0$ in $L^2$ and hence in probability.
    
    Now we turn our attention to 
    $B_n:=\Ex [T_n\mid K_n]-\Ex [T_n]$. For any constant $M$, let $\mu_n(M)=\Ex [T_n\mid K_n=M]=M(2n-M-1)/4$. Let $\Delta_n=K_n-a_n n$. The Taylor expansion of $\mu_n(K_n)$ is exact since $\mu_n(M)$ is quadratic. Hence,
    \[\mu_n(K_n)=\mu_n(a_n n)+\mu_n'(a_n n)\Delta_n-\frac14\Delta_n^2.\]
    Since $E[B_n]=0$, where $B_n$ is defined in Equation $\eqref{an and bn inversions}$, we have
    \[B_n=\mu_n'(a_n n)\Delta_n-\frac14\bigl(\Delta_n^2-\mathbb E\Delta_n^2\bigr).\]
    By Assumptions (2) and (3), we have $\Delta_n=O_p(\sqrt{n})$. Therefore $n^{-3/2}(\Delta_n^2-\Ex [\Delta_n^2])\to 0$ in probability.
    In addition, we have
    \[
    \frac{\mu_n'(a_n n)}{n}
    =\frac{2n-1}{4n}-\frac{a_n n}{2n}\longrightarrow\frac{1-a}{2}.
    \]
    Therefore, letting $Z\sim\mathcal{N}(0,\tau^2)$ be the distributional limit from Assumption (2), we deduce that
    \begin{equation}
    \label{certain convergence inversions}
        \frac{B_n}{n^{3/2}}
        =
        \Bigl(\frac{\mu_n'(a_n n)}{n}\Bigr)\cdot \frac{\Delta_n}{\sqrt n}
        +o_p(1)
        \Rightarrow_d\, \frac{1-a}{2}\,Z=_d\,\mathcal{N}\left(0,\frac{(1-a)^2\cdot\tau^2}{4}\right)
    \end{equation}
    where we used Assumptions (1) and (2). Now, notice that $B_n$ and $f(a_n)$ are independent because $K_n$ and the random variables $R_i$ are independent. We conclude that
    \[\frac{T_n-\mathbb \Ex T_n}{n^{3/2}}
        =
        \frac{A_n}{n^{3/2}}+\frac{B_n}{n^{3/2}}
        =
        \frac{f(a_n)}{n^{3/2}}+\frac{B_n}{n^{3/2}}+o_p(1)
        \Rightarrow_d\, \mathcal{N}\Bigl(0,\frac{1-(1-a)^3}{36}+\frac{(1-a)^2}{4}\tau^2\Bigr)
        \]
    by Proposition $\ref{deterministic convergence inversions}$ and Equation $\eqref{certain convergence inversions}$.
\end{proof}
We are now ready to prove the central limit theorem for inversions. In particular, the critical regime is a special case of Theorem $\ref{general inversions}$.
\begin{proof}[Proof of Theorem $\ref{inversions main}$]
    We have that the $r=cn$ shuffles case follows from Theorem $\ref{general inversions}$ applied to the random variables $K_n:=K_{cn,n}$. Note that the random variables $K_{cn,n}$ satisfy all three assumptions in Theorem $\ref{general inversions}$ by Proposition $\ref{facts about krn}$. The last step is to compute the corresponding variance $\sigma^2(a,\tau^2)$ from Theorem $\ref{general inversions}$. By Proposition $\ref{facts about krn}$, we set $a=1-e^{-c}$ and $\tau^2=e^{-c}-(1+c)e^{-2c}$. We have

    \begin{equation*}
        \begin{aligned}
            \sigma^2(1-e^{-c},\,e^{-c}-(1+c)e^{-2c}) &=\frac{1-e^{-3c}}{36}+\frac{e^{-2c}}{4}\left(e^{-c}-(1+c)e^{-2c}\right)\\
            &=\frac{1+8e^{-3c}-9(1+c)e^{-4c}}{36}.
        \end{aligned}
    \end{equation*}
    
    We next address the mixed regime. Let the random variables $R_i$ be independent discrete uniform, each supported on $\{0,1,2,\ldots,n-i\}$. By Theorem $\ref{decomposition of inversions}$, we can rewrite
    \[I_{r,n}=_d\,\sum_{i=1}^n R_i-\underbrace{\sum_{i=K_{r,n}+1}^n R_i}_{:=\,G_{r,n}}.\]
    We will find conditions on $r$ to make $n^{-3/2}\cdot G_{r,n}\to 0$ in probability. We have that
    \begin{equation}
    \label{cdl exp of gnr}
        \begin{aligned}
         \Ex[G_{r,n}\mid K_{r,n}]&=\sum_{i=K_{r,n}+1}^n \frac{n-i}{2}\\
        &=\frac{(n-K_{r,n})(n-K_{r,n}-1)}{4}.
        \end{aligned}
    \end{equation}
    By the law of iterated expectations, Equation $\eqref{cdl exp of gnr}$, and Proposition $\ref{facts about krn}$, we have that
    \[n^{-3/2}\cdot \Ex[G_{r,n}]=\frac{n-1}{4\sqrt{n}}\left(1-\frac{2}{n}\right)^r\leq \frac{\sqrt{n}\cdot\exp(-2r/n)}{4}.\]
    This implies that $n^{-3/2}\cdot G_{r,n}$ converges to zero in $L^1$, and therefore in probability, when there exists a sequence $c_n\to\infty$ such that $r\gtrsim (n\log n)/4 +c_n n$. By Theorem $\ref{slutsky}$, under these assumptions on $r$, we have that $I_{r,n}$ and the number of inversions of a uniformly random permutation have the same limiting distribution. The convergence in this case follows after applying Proposition $\ref{uniform case convergences}$. This concludes the proof of Theorem $\ref{inversions main}$.

\end{proof}
\section{Future Directions}
\label{Section 8}
We conclude by highlighting several directions for future research. To start, we recall that the expected value of the number of $i$-cycles of $r$ iterated top-to-random shuffles of an $n$-card deck, when $i\geq 2$, is equal to $i^{-1}\cdot (1-(1-i/n))^r$. This was found by Richard Stanley in unpublished work using the Robinson-Schensted-Knuth (RSK) correspondence and symmetric function theory, as summarized in Section 8.4 of Diaconis and Fulman \cite{MathofShufflingCards}. With representation-theoretic methods, Dominic Arcona and the author have recovered this expected value when $i=2$ (i.e. the expected number of transpositions) in unpublished work. In a similar direction, Arcona \cite{arcona} and Fulman \cite{fulmanIcycles} both used representation theory to obtain the limiting distributions of cycle counts after iterated star transposition shuffles and iterated random $i$-cycle shuffles, respectively. This suggests that it would be of interest to establish a (joint) central limit theorem for the cycle counts of iterated random-to-top shuffles. 

From the standpoint of refining our results, it would be very interesting to obtain convergence rates in Theorem  $\ref{fixed points main}$, Theorem $\ref{descents main}$, and Theorem $\ref{inversions main}$. Stein's method, as surveyed by Ross \cite{Ross}, is a useful way to analyze rates of convergence for statistics; however, it is not immediately clear how to adapt it to the setting of random-to-top shuffles. One could also analyze the $n\ll r\ll (n\log n)/2$ regime for descents and the analogous $n\ll r \ll (n\log n)/4$ regime for inversions of iterated random-to-top shuffles. It may be the case that both statistics mix in $\omega(n)$ shuffles.

Examining the statistics of biased types of shuffling is another intriguing direction. The mixing time of biased random-to-top shuffling, which is also known as the Tsetlin Library, has been studied in some special cases \cite{biasedRTT}. The Tsetlin library also has fascinating enumerative properties; see Chatterjee, Diaconis, and Kim \cite{ChatterjeeDiaconisKim}. A natural extension would be to consider fixed points, inversions, and descents for iterated shuffles of this kind. 

Finally, we note that random-to-top shuffling and sampling without replacement are intrinsically connected. This is because the stationary distribution of the Tsetlin Library is the Luce distribution on permutations, which is a type of biased sampling without replacement. The Luce model has seen interest in the past year with a central limit theorem for the number of inversions shown in \cite{BorgaChatterjeeDiaconis}. Importantly, determining the limiting distribution(s) of the number of fixed points in the Luce model is an open question of the authors in \cite{BorgaChatterjeeDiaconis}. We would be delighted if the techniques in this paper could be used to make any progress in these areas.

\section*{AI Disclosure}
The author acknowledges that both ChatGPT 5.2-5.3 and Google Gemini were used under licenses from the University of Southern California (USC) to assist with the following tasks on both the first version and the current version of this article.
\begin{enumerate}
    \item Write Python code for the simulations in Figure $\ref{fig: fixed points}$, Figure $\ref{fig:descents-comparison}$, and Figure $\ref{fig:inversions-comparison}$.
    \item Test the distributional equalities in Section $\ref{Section 4}$ and test the combinatorial proof of Equation $\eqref{expected fixed points Pehlivan}$, both by simulating random-to-top shuffles.
    \item Search for relevant literature.
    \item Identify typos and grammar mistakes in previous drafts of this paper.
\end{enumerate}
Readers interested in viewing the corresponding conversations with each AI tool are kindly asked to contact the author.
\bibliographystyle{amsplain}
\bibliography{AC}

@book{durrett2019probability,
  title     = {Probability: Theory and Examples},
  author    = {Durrett, R.},
  edition   = {5},
  year      = {2019},
  publisher = {Cambridge University Press},
  address   = {Cambridge}
}

@article{m-dependentclt-hoeffding,
title = {The central limit theorem for dependent random variables},
author = {Hoeffding, W. and Robbins, H.},
journal = {Duke Math. Journal},
year = {1948},
volume = {15},
pages = {773-780}
}

@book{MathofShufflingCards,
author = {Diaconis, P. and Fulman, J.},
title = {The Mathematics of Shuffling Cards},
publisher = {AMS},
year = {2023},
}

@article{Margolius,
author = {Margolius, B.},
title = {Permutations with Inversions},
journal = {Journal of Integer Sequences},
volume = {4},
year = {2001}
}

@article{ArratiaTavare1992,
  title        = {The Cycle Structure of Random Permutations},
  author       = {Arratia, R. and Tavar{\'e}, S.},
  journal      = {Annals of Probability},
  volume       = {20},
  number       = {3},
  pages        = {1567--1591},
  year         = {1992},
  doi          = {10.1214/aop/1176989707}
}

@article{ChatterjeeDiaconisDescents,
author = {Chatterjee, S. and Diaconis, P.},
title = {A central limit theorem for a new statistic on permutations},
journal = {Indian Journal of Pure and Applied Math.},
year = {2018},
pages = {561-573},
volume = {48},
}

@article{FulmanDescents,
author = {Fulman, J.},
title = {Stein’s method and non-reversible Markov chains},
journal = {IMS Lecture Notes Monogr. Ser.},
year = {2004},
pages = {66-74}
}

@book{AsymptoticStatistics,
author = {van der Vaart, A.W.},
title = {Asymptotic Statistics},
publisher = {Cambridge University Press},
year = {1998}}

@article{Weiss1958Limiting,
  author       = {Weiss, I.},
  title        = {Limiting Distributions in Some Occupancy Problems},
  journal      = {The Annals of Mathematical Statistics},
  year         = {1958},
  volume       = {29},
  number       = {3},
  pages        = {878-884},
  month        = {September},
  doi          = {10.1214/aoms/1177706544},
}

@article{Ferrante,
author = {Ferrante, M. and Saltalamacchia, M.},
title = {The Coupon Collector's Problem},
journal = {Materials Matem\`atics},
volume = {2014},
}

@article{Pehlivan,
author = {Pehlivan, L.},
title = {On top to random shuffles, no feedback card guessing, and fixed points of permutations},
year = {2009},
note = {PhD Thesis}
}

@misc{fulmanIcycles,
author = {Fulman, J.},
title = {Fixed points of non-uniform permutations and representation theory of the symmetric group},
note = {arXiv preprint},
year = {2024},
howpublished = {\url{https://arxiv.org/pdf/2406.12139}}
}

@misc{arcona,
author = {Arcona, D.},
title = {Representation theory and cycle statistics for random walks on the symmetric group},
note = {arXiv preprint},
year = {2025},
howpublished = {\url{https://arxiv.org/pdf/2512.13969}},
}

@article{Ross,
author = {Ross, N.},
title = {Fundamentals of Stein's Method},
journal = {Probability Surveys},
volume = {8},
year = {2011},
pages = {210-293}}

@article{ChatterjeeDiaconisKim,
author = {Chatterjee, S. and Diaconis, P. and Kim, G.},
title = {Enumerative theory for the Tsetlin library},
journal = {Journal of Algebra},
volume = {655},
year = {2024},
pages = {139-162}}

@article{biasedRTT,
author = {Jonasson, J.},
title = {Biased random-to-top shuffling},
journal = {Ann. Appl. Prob.},
year = {2006},
volume = {16},
pages = {1034-1058}}

@misc{BorgaChatterjeeDiaconis,
author = {Borga, J. and Diaconis, P. and Chatterjee, S.},
title = {Permuton and local limits for the Luce model},
year = {2025},
note = {arXiv preprint},
howpublished = {\url{https://arxiv.org/pdf/2509.07729}}}

@article{KimLee,
    title = {Central limit theorem for descents in conjugacy classes of $S_n$},
    author = {Kim, G. and Lee, S.},
    year = {2020},
    journal = {Journal of Combinatorial Theory, Series A},
    volume = {169}

}

@article{LiuYin,
    title = {Descents and Flag Major Index on Conjugacy
Classes of Colored Permutation Groups Without
Short Cycles},
    author = {Liu, K. and Yin, M.},
    journal = {Electronic Journal of Combinatorics},
    year = {2025},
    volume = {32},
    pages = {3-47}
    }

@article{FulmanDescents1998,
    title = {The Distribution of Descents in Fixed Conjugacy Classes of the Symmetric Groups},
    author = {Fulman, J.},
    journal = {Journal of Combinatorial Theory, Series A},
    year = {1998},
    pages = {171-180},
    volume = {84}}

@misc{loth2023permutationstatisticsconjugacyclasses,
      title={Permutation Statistics in Conjugacy Classes of the Symmetric Group}, 
      author={Loth, J.C. and Levet, M. and Liu, K. and Stucky, E.N. and Sundaram, S. and Yin, M.},
      year={2023},
      note ={arXiv preprint},
      howpublished = {\url{https://arxiv.org/abs/2301.00898}}
}

@article{descentsmatchings,
    title = {Distribution of Descents in Matchings},
    author = {Kim, G.},
    journal = {Annals of Combinatorics},
    volume = {23},
    pages = {73-87},
    year = {2019}
}

@article{AthDiaconis,
title = {Functions of random walks on hyperplane arrangements},
pages = {410-437},
journal = {Adv. Appl. Math.},
year = {2010},
author = {Athanasiadis, C. and Diaconis, P.},
volume = {45}
}

@article{ShufflingStopping,
author = {Aldous, D. and Diaconis, P.},
title = {Shuffling cards and stopping times},
journal = {Amer. Math. Monthly},
pages = {333-348},
year = {1986},
volume = {93}}

@article{AnalysisTopToRandom,
author = {Diaconis, P. and Fill, J.A. and Pitman, J.},
title = {Analysis of Top to Random Shuffles},
journal = {Combinatorics, Probability and Computing},
volume = {1},
pages = {135-155},
year = {1992}
}

@article{BayerDiaconisRiffleShuffle,
title = {Trailing the Dovetail Shuffle to its Lair},
author = {Bayer, D. and Diaconis, P.},
journal = {Ann. Appl. Probab.},
year = {1992},
pages = {294-313},
volume = {2}
}

@article{Schramm,
title = {Compositions of random transpositions},
journal = {Israel Journal of Mathematics},
year = {2005},
pages = {221-243},
author = {Schramm, O.},
volume = {147}}
\end{document}